\pdfoutput=1
\documentclass[reqno]{amsart}
\usepackage{amsmath,calligra,mathrsfs}
\usepackage{amsthm}
\makeatletter
\def\els@aparagraph[#1]#2{\elsparagraph[#1]{#2\@addpunct{.}}}
\def\els@bparagraph#1{\elsparagraph*{#1\@addpunct{.}}}
\makeatother
\usepackage{amssymb}
\usepackage[left=1.25in, right=1.25in]{geometry}
\usepackage{tikz}[2010/10/13]
\usetikzlibrary{decorations.pathreplacing}
\usepackage[all]{xy}
\usepackage{graphicx}
\usepackage{indentfirst}
\usepackage{mathtools}
\usepackage{bm}
\usepackage{bbold}
\usepackage{bbm}
\usepackage{mathrsfs}
\usepackage{latexsym}
\usepackage{hyperref}
\usepackage{microtype}
\usepackage{float}
\usetikzlibrary{shapes.geometric}
\usepackage{ifthen}

\DeclareMathOperator{\Hom}{Hom}
\DeclareMathOperator{\vect}{\textbf{Vec}}
\DeclareMathOperator{\GHZ}{GHZ}
\DeclareMathOperator{\I}{I}
\DeclareMathOperator{\II}{II}
\DeclareMathOperator{\III}{III}

\begin{document}
	\title{Universal skein theory for group actions}
	\author{Yunxiang Ren}
	\address{Department of Mathematics and Department of Physics\\ Harvard University}
	\email{yren@g.harvard.edu}
	\maketitle
	\newtheorem{Lemma}{Lemma}
	\theoremstyle{plain}
	\newtheorem{theorem}{Theorem~}[section]
	\newtheorem{main}{Main Theorem~}
	\newtheorem{lemma}[theorem]{Lemma~}
	\newtheorem{proposition}[theorem]{Proposition~}
	\newtheorem{corollary}[theorem]{Corollary~}
	\newtheorem{definition}[theorem]{Definition~}
	\newtheorem{notation}[theorem]{Notation~}
	\newtheorem{example}[theorem]{Example~}
	\newtheorem{assumption}[theorem]{Assumption~}
	\newtheorem*{remark}{Remark~}
	\newtheorem*{cor}{Corollary~}
	\newtheorem{question}{Question}
	\newtheorem*{claim}{Claim}
	\newtheorem*{conjecture}{Conjecture~}
	\newtheorem*{fact}{Fact~}
	\renewcommand{\proofname}{\bf Proof}

\begin{abstract} Given a group action on a finite set, we define the group-action model which consists of tensor network diagrams which are invariant under the group symmetry. In particular, group-action models can be realized as the even part of group-subgroup subfactor planar algebras. Moreover, all group-subgroup subfactor planar algebras arise in this way from transitive actions. In this paper, we provide a universal skein theory for those planar algebras. With the help of this skein theory, we give a positive answer to a question asked by Vaughan Jones in the late nineties. 
	
\end{abstract}
\section{Introduction}
In the late nineties, Vaughan Jones asked the following question.
\begin{question}[Vaughan Jones, \cite{JonPC}]\label{quest: VJ}
		Is the group-subgroup subfactor planar algebra for $S_2\times S_3\subset S_5$ generated by its $2$-boxes?
\end{question}
In this paper, we give a positive answer to Question \ref{quest: VJ} by exploring \textit{group-action models} in the framework of tensor networks. Tensor network theory has played a crucial role during recent years in modern quantum physics and led to numerous results in understanding quantum many-body systems and quantum information theory. In the framework of tensor networks, physical states are represented by diagrams, called \textit{tensor network diagrams}. The basic building blocks are \textit{tensors} associated with a $d$-dimensional vector space $V$ and a tensor network diagram is obtained from the tensors by applying the three basic operations: \textit{tensor product}, \textit{contraction} and \textit{permutation} (see Definitions \ref{def: tensor networks} and \ref{def: tensor network diagrams}). The tensor networks are closely related to the theory of tensor categories \cite{EGNO}: Let $\vect$ be the monoidal category of finite-dimensional vector spaces. Every tensor network diagram can be realized as a morphism in the monoidal category $\vect$. Moreover, the $\GHZ$ tensor (see Definition \ref{def: GHZ}) provides a Frobenius algebra structure. In this paper, we study the tensor network diagrams with group symmetry: Suppose $V$ is equipped with a unitary action of a finite group $G$. This action can be extended on all the tensor network diagrams. The \textit{group-action model} is defined to be the collection of all diagrams which are invariant under this group action. 

Now we explain Question \ref{quest: VJ} in terms of group-action models. First of all, the even part of the group-subgroup subfactor planar algebra for $S_2\times S_3\subset S_5$ is a group-action model associated with the \textit{Petersen graph} (see Definition \ref{def: Petersen graph}). An $n$-box in the even part of the subfactor planar algebra is a rank-$n$ tensor in the group-action model. The phrase ``generated by" means ``obtained by applying planar tangles on". Therefore, in the context of group-action models, Question \ref{quest: VJ} is equivalent to the following, 
\begin{question}\label{ques: equiv formuation of ques of VJ}
	Can every tensor network diagram in the group-action model associated withg the Petersen graph be obtained by applying operations of tensor product and contraction on rank-$2$ tensors and the $\GHZ$ tensor?
\end{question}
The main tool we use to answer this question is \textit{skein theory} of group-action models. One of the fundamental questions in tensor network theory is to evaluate tensor network diagrams. To be precise, one needs to determine the \textit{skein theory} of the collection of tensor network diagrams under given conditions, such as invariance under some group symmetry in this paper. A skein theory is a collection of generators, relations and an evaluation algorithm: generators are the tensors on which every tensor network diagram can be obtained; relations are equalities between tensor network diagrams; an evaluation algorithm is provided such that every closed tensor network diagram can be evaluated into a scalar. In general, a basis for the space of tensor network diagrams with $n$ boundary points needs to be specified. These are called standard forms. Every tensor network diagram can thereby be evaluated to a linear combination of standard forms. With a skein theory in hand, one can completely describe the tensor networks. 

In this paper, we answer Question \ref{ques: equiv formuation of ques of VJ} in the following way. In \S \ref{sec: The universal skein theory for group action}, we provide a universal skein theory for group-action models as in Theorem \ref{thm: skein theory} for an action of a finite group $G$ on the $d$-dimensional space $V$. The generators are the $\GHZ$ tensor, a rank-$4$ tensor, called \textit{transposition}, and a rank-$d$ tensor, called \textit{molecule}. Moreover, the even parts of all group-subgroup subfactor planar algebras arise as group-action models. Therefore, we provide a universal skein theory for group-subgroup subfactor planar algebras (see Corollary \ref{cor: group-subgroup}). With the help of the universal skein theory, we give a positive answer to \ref{ques: equiv formuation of ques of VJ} by providing an explicit construction of the generators using the $\GHZ$ tensor and rank-$2$ tensors. In this case, the molecule is a rank-$10$ tensor and it is constructed as the Petersen graph while considering the vertices as the $\GHZ$ tensor and the edges as a rank-$2$ tensor as explained in Theorem \ref{thm: 2-box generating}. The rank-$2$ tensor behaves like an atom and the $\GHZ$ tensor describes the \textit{bonding} between different atoms. In nature, atoms are more basic building blocks but materials appear as in a composition of molecules. Similarly, tensor network diagrams in this group-action model are generated by rank-$2$ tensors, \textit{atoms} and the $\GHZ$ tensor, the \textit{bonding}, but a more natural way to describe the diagrams in terms of the rank-$10$ tensor, the \textit{molecule}. The major difficulty to answer Question \ref{ques: equiv formuation of ques of VJ} is to construct the rank-$4$ tensor \textit{transposition}. In general, one needs to show that the space of rank-$4$ tensors is spanned by the $\GHZ$ tensor and rank-$2$ tensors. In this case, the space of rank-$4$ tensors is $107$-dimensional and thus it requires an explicit construction of $107$ linearly independent diagrams generated by the $\GHZ$ tensor and rank-$2$ tensors. In Theorem \ref{thm: condition1}, we provide a construction by utilizing the combinatorial data from the group action of $S_5$ on the Petersen graph.  
\bigskip
\paragraph{\textbf{Acknowledgements}} The author would like to thank Vaughan Jones for support, encouragement, and many inspiring conversations. The author would like to thank Zhengwei Liu for many helpful discussions and also Youwei Zhao for designing a Javascript program for generating the diagrams in the Appendix. The research was supported by NSF Grant DMS-1362138 and TRT 0159 from the Templeton Religion Trust.
\section {Background}
Question \ref{quest: VJ} traces back to the program of classifying subfactor planar algebras which originated from modern subfactor theory initiated by Vaughan Jones \cite{Jon83}. Subfactor planar algebras were introduced as an axiomatization of \textit{standard invariants} for subfactors \cite{Jon99}. There are three different approaches to classify subfactor planar algebras, namely, classification by index, by simple generator and relations, and by skein theory. The approach of classification by index appeared first historically and has been extremely successful, see \cite{JMS14}. In 1999, Bisch and Jones proposed the program of classifying subfactor planar algebras by simple generators and relations \cite{Jon99, BisJon00}. In this scheme, the generator is a non-Temperley-Lieb $2$-box, namely a diagram with $4$ boundary points. The dimension of the $3$-box space is required to be small. Such planar algebras are called singly generated planar algebras of small dimension. The restriction of small dimension immediately forces the appearance of certain skein relations. When the dimension of the $3$-box space is less than or equal to $14$, these relations are sufficient to provide a skein theory and thus to classify subfactor planar algebras \cite{BisJon03,BJL17}. In this approach, the index is no longer a restriction and one can obtain subfactor planar algebras with large indices in the classification. In the categorical setting, Morrison, Peters and Synder investigated trivalent categories and gave a classification of small dimensions \cite{MPS15}. In general, the dimension restriction may not be sufficient to provide a skein theory to describe the planar algebras. 

Another approach to classify subfactor planar algebras is by skein theory. It has no restrictions on either indices nor dimensions. From this point of view, the classification of Bisch and Jones in \cite{BisJon00,BisJon03} can be considered as a classification of exchange relation planar algebras generated by a single 2-box, where the exchange relation was introduced by Landau \cite{Lau02} generalizing Bisch's exchange relation of biprojections \cite{Bis94}. Liu gave the classification of exchange relation planar algebras on two generators, and completely classified commute relation planar algebras with multiple generators without dimension restrictions in \cite{Liu16}. Liu also gave a complete classification of singly-generated Yang-Baxter relation planar algebra in \cite{Liu15}, where a surprising new family of subfactor planar algebras was constructed. In \cite{JLR}, Corey Jones, Liu and the author gave a classification of planar algebras generated by a $3$-box satisfying a relation introduced by D. Thurston in \cite{Thu04}, which we called the Thurston relation.

It is very interesting that the group-subgroup subfactor planar algebra for $S_2\times S_3\subset S_5$ has the same dimensions on $0,1,2,3$-box spaces with the Birman-Murakami-Wenzl (BMW) planar algebras. Vaughan Jones asked a natural question whether this planar algebra is generated by its 2-boxes, namely Question \ref{quest: VJ}, in the late nineties. We give a positive answer to this question. The follow-up task is to find its skein theory since it does not satisfy the Yang-Baxter relations as the BMW planar algebras. Indeed, the skein relations forced by the dimension restriction on the $3$-box space, which we call the One-way Yang-Baxter relation, are not sufficient to determine the planar algebra. In this paper, we discover a natural skein theory for this planar algebra. It is surprising that one of the generators is a $10$-box and one essential relation appears in the $20$-box space from the perspective of skein theory. This phenomenon shows that the simple generators and relations of a planar algebra may live in larger box spaces, which is different from previous approaches in skein theory. 
\section{Preliminary}\label{sec: Prelimenary}
In this section, we recall some basics about tensor networks and introduce group-action models.
\begin{definition}[Spin models]\label{def: tensor networks}
	 Suppose $V$ is a $d$-dimensional vector space with a basis $\{e_1,e_2,\cdots,e_d\}$. The spin model $\mathscr{P}_\bullet$ associated with V is defined to be a family of vector spaces $\{\mathscr{P}_{n}\colon n\geq0\}$, where $\mathscr{P}_{n}=V^{\otimes n}$ for $n>0$, and $\mathscr{P}_0$ is the ground field $\mathbb{C}$. A rank-$n$ tensor is a vector in $\mathscr{P}_{n}$. We denote the basic tensor $e_{i_1}\otimes e_{i_2}\otimes \cdots \otimes e_{i_n}$ by $\vec{i}=(i_1,i_2,\cdots,i_n)$. Moreover, there are three operations on $\mathscr{P}_\bullet\colon$ For two basic tensor $\vec{i}=(i_1,i_2,\cdots,i_n)$ and $\vec{j}=(j_1,j_2,\cdots,j_m)$, we have that
	 \begin{itemize}
	 	\item Tensor product$\colon$
	 	\begin{equation}\label{equ: tensor product}
	 	\vec{i}\otimes \vec{j}=(i_1,i_2,\cdots,i_n,j_1,j_2,\cdots,j_m).
	 	\end{equation}
	 	\item Contraction$\colon$ for $1\leq k\leq n$,
	 	\begin{equation}\label{equ: contraction}
	 	C_{k,k+1}(\vec{i})=\delta_{i_k,i_{k+1}}\cdot (i_1,i_2,\cdots,i_{k-1},i_{k+2},\cdots,i_n),
	 	\end{equation}
	 	where $\delta$ is the Kronecker delta.
	 	\item Permutation$\colon$ for $1\leq k\leq n$,
	 	\begin{equation}\label{equ: permutation}
	 	S_{k,n}(\vec{i})=(i_1,i_2,\cdots,i_{k-1},i_{k+1},i_k,i_{k+2},\cdots,i_n).
	 	\end{equation} 
	 \end{itemize}
\end{definition}

\begin{definition}[Tensor network diagrams]\label{def: tensor network diagrams}
	Suppose $V$ is a $d$-dimensional vector space and $\mathscr{P}_\bullet$ is the spin model associated with $V$. Let $\Hom(n,m)$ be the space of linear transformations from $V^{\otimes n}$ to $V^{\otimes m}$. For $T\in\Hom(n,m)$, we represent it as follows,
	
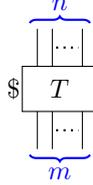
\begin{figure}[H]
		\begin{tikzpicture}
		\draw (.2,1.1)--(.2,-.5);
		\draw (.4,1.1)--(.4,-.5);
		\draw (.8,1.1)--(.8,-.5);
		\draw [dotted, thick](.45,-.25)--(.75,-.25);
		\draw [dotted, thick](.45,.85)--(.75,.85);
		\draw [fill=white] (0,0) rectangle (1,.6);
		\node at (.5,.3) {$T$};
		\node at (.1,.3) [left] {$\$$};
		\draw[decoration={brace,mirror,raise=5pt},decorate, thick, blue]
		(.1,-.4) -- node[below=6pt] {$m$} (.9,-.4);
		\draw[decoration={brace,raise=5pt},decorate, thick, blue]
		(.1,1) -- node[above=6pt] {$n$} (.9,1);
		\end{tikzpicture}
	\end{figure}
    Since $V$ is a finite-dimensional vector space, the dual space $V^*$ is naturally identified with $V$. Therefore, the space $\Hom(n,m)$ can be identified with $\Hom(0,n+m)$. In particular, for a rank-$n$ tensor $x\in\mathscr{P}_n$, one can represent it as follows:
    
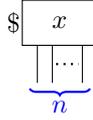
\begin{figure}[H]
    	\begin{tikzpicture}
    	\draw (0,0) rectangle (1,.6);
    	\node at(.5,.3) {$x$};
    	\node at (.1,.3) [left] {$\$$};
    	\draw (.2,0)--(.2,-.5);
    	\draw (.4,0)--(.4,-.5);
    	\draw (.8,0)--(.8,-.5);
    	\draw [dotted, thick](.45,-.25)--(.75,-.25);
    	\draw[decoration={brace,mirror,raise=5pt},decorate, thick, blue]
    	(.1,-.4) -- node[below=6pt] {$n$} (.9,-.4);
    	\end{tikzpicture}
    \end{figure}
\end{definition}

\begin{notation}\label{not: thick strings}
	For $n\in\mathbb{N}$, we use a thick string with $n$ labeled on the side to represent $n$ parallel strings as follows,
	\begin{equation}
	\begin{tikzpicture}
	\draw [ultra thick] (-1,0)--(-1,.5);
	\node at (-.95,.25)[left] {$n$};
	\node at (-.5,.25) {$=$};
	\draw [ultra thick] (0,0)--(0,.5);
	\node at (-.05,.25) [right] {$n$};
	\node at (.6,.25) {$=$};
	\draw (1,0)--(1,.5);
	\draw (1.2,0)--(1.2,.5);
	\draw (1.6,0)--(1.6,.5);
	\draw [dotted, thick](1.25,.25)--(1.55,.25);
	\draw[decoration={brace,mirror,raise=5pt},decorate, thick, blue]
	(1,.1) -- node[below=6pt] {$n$} (1.6,.1);
	\end{tikzpicture}
	\end{equation}
\end{notation}

With Notation \ref{not: thick strings}, Tensor product \eqref{equ: tensor product} and Contraction \eqref{equ: contraction} in Definition \ref{def: tensor networks} are represented as follows:
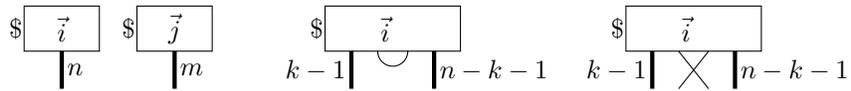
\begin{figure}[H]
	\begin{tikzpicture}
		\draw (0,0) rectangle (1,.6);
	\draw (1.5,0) rectangle (2.5,.6);
	\node at (.1,.3) [left] {$\$$};
	\node at (1.6,.3) [left] {$\$$};
	\node at (.5,.3) {$\vec{i}$};
	\node at (2,.3) {$\vec{j}$};
	\draw [ultra thick] (.5,0)--(.5,-.5);
	\node at (.45,-.25) [right] {$n$};
	\draw [ultra thick] (2,0)--(2,-.5);
	\node at (1.95,-.25) [right] {$m$};
	   \draw (4+0,0) rectangle (4+1.8,.6);
	\draw (4+.7,0) arc [radius=.2,start angle=180, end angle=360];
	\draw [ultra thick](4+.35,0)--(4+.35,-.5);
	\draw [ultra thick] (4+1.45,0)--(4+1.45,-.5);
	\node at (4+.4,-.25)[left] {$k-1$};
	\node at (4+1.4,-.25) [right] {$n-k-1$};
	\node at (4+.1,.3) [left] {$\$$};
	\node at (4+.8,.3) {$\vec{i}$};
		\draw (8+0,0) rectangle (8+1.8,.6);
	\draw (8+.7,0)--(8+1.1,-.5);
	\draw (8+1.1,0)--(8+.7,-.5);
	\draw [ultra thick](8+.35,0)--(8+.35,-.5);
	\draw [ultra thick] (8+1.45,0)--(8+1.45,-.5);
	\node at (8+.4,-.25)[left] {$k-1$};
	\node at (8+1.4,-.25) [right] {$n-k-1$};
	\node at (8+.1,.3) [left] {$\$$};
	\node at (8+.8,.3) {$\vec{i}$};
	\end{tikzpicture}\caption{Tensor product, contraction, permutation.}
\end{figure}
Therefore, a rank-$n$ tensor is represented by a labeled rectangle with $n$ strings attached and a dollar sign $\$$ on the left side ( which will be omitted when no confusion) and a tensor network diagram is an element in $\mathscr{P}_\bullet$. From the perspective of tensor category theory, a rank-$n$ tensor can be realized as a morphism from the unit object $\mathbb{1}\cong\mathbb{C}$ to the object $V^{\otimes n}$ in the symmetric monoidal category $\vect$. The tensor product of two tensors comes naturally from the tensor functor $\otimes$. The contraction is realized by the \textit{coevaluation} and the self-duality of the object $V$. The permutation is realized by the symmetric braiding $C_{V,V}\in\Hom(V\otimes V, V\otimes V)$. Note that the symmetric braiding can be viewed as a morphism in $\Hom(\mathbb{1},V^{\otimes 4})$, namely, a rank-$4$ tensor. Therefore, we have that
\begin{definition}\label{def: transposition}
	The transposition $R$ is an element in $\mathscr{P}_{4}$ defined as follows: 
	\begin{equation}
	R=\sum_{i,j=1}^d(i,j,i,j).
	\end{equation}
	Furthermore, the transposition $R$ is represented by the following diagram:
	
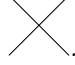
\begin{figure}[H]
		\begin{tikzpicture}
		\draw (0,0)--(.8,.8);
		\draw (0,.8)--(.8,0);
		\end{tikzpicture}.
	\end{figure}
\end{definition}
Moreover, the object $V$ is a Frobenius algebra object in the symmetric monoidal category $\vect$ provided with the so-called $\GHZ$ tensor as follows.
\begin{definition}\label{def: GHZ}
	The \text{Greenberger–Horne–Zeilinger $(\GHZ)$} tensor is a rank-$3$ tensor in $\mathscr{P}_3$ defined as follows,
	\begin{equation}
	\GHZ=\sum_{j=1}^{d} (j,j,j).
	\end{equation}
	Furthermore, the $\GHZ$ tensor is represented by the following diagram:
	
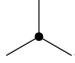
\begin{figure}[H]
		\begin{tikzpicture}
		\foreach \i in {1,...,3}
		{
			\pgfmathsetmacro{\a}{cos(360*(\i-1)/3+90)};
			\pgfmathsetmacro{\b}{sin(360*(\i-1)/3+90)};
			\draw(0,0)--(.5*\a,.5*\b);
		}
		\draw [fill=black] (0,0) circle [radius=.05];
		\end{tikzpicture}.
	\end{figure} 
\end{definition}
The relations for a Frobenius algebra is presented in Proposition \ref{pro: relation for GHZ}. Therefore, the spin model $\mathscr{P}_\bullet$ can be realized as the even part of the spin model planar algebra \cite{Jon99}. In this paper, we investigate submodels of spin models which are invariant under group symmetry. 
\begin{definition}[Group-action models]\label{def: groupaction models}
	Let $V$ and $\mathscr{P}_\bullet$ be as in Definition \ref{def: tensor networks}. Suppose $G$ is a subgroup of the symmetric group $S_d$. The group $G$ has a natural action on the vector space $V$ which can be extended to $V^{\otimes n}$ diagonally for every $n\in\mathbb{N}$. Moreover, the extended action of $G$ commutes with tensor product \eqref{equ: tensor product}, contraction \eqref{equ: contraction} and the permutation \eqref{equ: permutation}, namely, this induces an action of the group $G$ on the spin model $\mathscr{P}_\bullet$. Therefore the group-action model, denoted by $\mathscr{P}_\bullet^G$, is the fixed-points of the group action, i.e.,  
	\begin{equation}\label{equ: def of group action model}
	\mathscr{P}^G_n=\{x\in \mathscr{P}_n\colon g\cdot x=x,\forall g\in G\}.
	\end{equation}
\end{definition}
\begin{proposition}\label{pro: GHZ and R for group action}
	The GHZ tensor and the transposition $R$ belong to $\mathscr{P}_\bullet^G$ for any group $G\leq S_d$. 
\end{proposition}
\begin{proof}
	By definition, we have that for $g\in G$
	\begin{equation*}
	g\cdot\GHZ=\sum_{j=1}^d g\cdot (j,j,j)=\sum_{j=1}^d (g\cdot j,g\cdot j,g\cdot j)=\sum_{j=1}^{d} (j,j,j)=\GHZ.
	\end{equation*}
	Therefore, $\GHZ\in\mathscr{P}_3^G$ and similarly $R\in\mathscr{P}_4^G$. 
\end{proof}
By Proposition \ref{pro: GHZ and R for group action}, a group-action model $\mathscr{P}^G_\bullet$ can be realized as the even part of a planar algebra. The corresponding planar algebra is described in \cite{Jon12}. If the group action is transitive, the planar algebra is the group-subgroup subfactor planar algebra for $H\leq G$, where $H$ is the stabilizer of a single point. The $\GHZ$ tensor is indeed a Frobenius algebra as proved in Proposition \ref{pro: relation for GHZ} and Corollary \ref{cor: GHZ}. In the framework of planar algebras, the $\GHZ$ tensor is represented as follows:
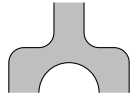
\begin{figure}[H]
	\begin{tikzpicture}
	\path [fill=lightgray] (-.4,0)--(-.4,.4) arc[radius=.2, start angle=180, end angle=90]--(0,.6) arc[radius=.2, start angle=270, end angle=360]--(.2,1.2)--(.6,1.2)--(.6,.8) arc[radius=.2, start angle=180, end angle=270]--(1,.6) arc[radius=.2,start angle=90, end angle=0]--(1.2,0)--(.8,0) arc [radius=.4, start angle=0, end angle=180];
	\draw (0,0) arc [radius=.4, start angle=180, end angle=0];
	\draw (-.4,0)--(-.4,.4) arc[radius=.2, start angle=180, end angle=90]--(0,.6) arc[radius=.2, start angle=270, end angle=360]--(.2,1.2);
	\draw (1.2,0)--(1.2,.4) arc[radius=.2, start angle=0, end angle=90]--(.8,.6) arc[radius=.2, start angle=270, end angle=180]--(.6,1.2);
	\end{tikzpicture}
\end{figure}
Therefore, the relations in Proposition \ref{pro: relation for GHZ} and Corollary \ref{cor: GHZ} turn to purely planar isotopy. Moreover, a rank-$n$ tensor $x$ in Definition \ref{def: tensor network diagrams} is represented by the following diagram
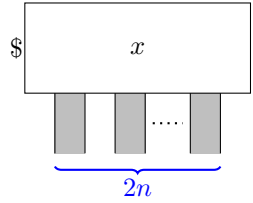
\begin{figure}[H]
	\begin{tikzpicture}
	\path [fill=lightgray](.4,0) rectangle (.8,-.8);
	\path [fill=lightgray] (1.2,0) rectangle (1.6,-.8);
	\path [fill=lightgray] (2.6,0) rectangle (2.2,-.8);
	\draw (0,0) rectangle (3,1.2);
	\node at (1.5,.6) {$x$};
	\node at (.1,.6)[left] {$\$$};
	\draw (.4,0)--(.4,-.8);
	\draw (.8,0)--(.8,-.8);
	\draw (1.6,0)--(1.6,-.8);
	\draw (1.2,0)--(1.2,-.8);
	\draw (2.2,0)--(2.2,-.8);
	\draw (2.6,0)--(2.6,-.8);
	\draw [thick, dotted] (1.7,-.4)--(2.1,-.4);
	\draw[decoration={brace,mirror,raise=5pt},decorate, thick, blue]
	(.4,-.8) -- node[below=6pt] {$2n$} (2.6,-.8);
	\end{tikzpicture}
\end{figure}
Here, a single string is doubled to be two parallel strings and we put a checkerboard shading on the regions separated by those strings such that the region where the dollar sign sits in is unshaded. Such diagrams form the even part of the subfactor planar algebra. Besides, one can have diagrams whose dollar sign sits in a shaded region and such diagrams form the odd part of the subfactor planar algebra. The even part and odd one are related diagrammatically by the $1$-click rotation which is known as the Fourier transform on subfactors.
\section{The universal skein theory for group-action models}\label{sec: The universal skein theory for group action}
In this section, we discuss group-action models as in Definition \ref{def: groupaction models} and provide a universal skein theory for such models in the sense that there is a group-parametrized family of skein theories which completely describes the group-action models.
\begin{assumption}
	Let $V$ be a vector space of dimension $d$ and $G$ be a subgroup of the symmetric group $S_d$ and the group-action model associated with $G$ by $\mathscr{P}_\bullet^G$. In this case, we have the circle parameter as follows:
	\begin{equation}\label{equ: circle parameter}
	\begin{tikzpicture}
	\draw (0,0) circle [radius=.3];
	\node at (.7,0) {$=d.$};
	\end{tikzpicture}
	\end{equation}
\end{assumption}
By the definition of group-action models, the space $\mathscr{P}^G_{n}$ is spanned by the orbits of basic tensors under the action of $G$ for every $n\in\mathbb{N}$. Therefore, we introduce the following notation:
\begin{notation}
	For $n\in\mathbb{N}$ and a basic rank-$n$ tensor $\vec{i}=(i_1,i_2,\cdots,i_n)\in\mathscr{P}_{n}$ 
	\begin{equation}
	[\vec{i}]=[i_1,i_2,\cdots,i_n]=\sum_{g\in G}(g\cdot i_1,g\cdot i_2,\cdots,g\cdot i_n).
	\end{equation}
\end{notation}
From Proposition \ref{pro: GHZ and R for group action}, we know that $\mathscr{P}^G_\bullet$ contains the $\GHZ$ tensor and the transposition $R$. It is shown in \cite{Jon94} that the group-action model $\mathscr{P}^{S_d}_\bullet$ is generated by $\GHZ$ and $R$. For any subgroup $G$ of $S_d$, we have that $\mathscr{P}^{S_d}_\bullet\subset \mathscr{P}_\bullet^G$. Therefore, it is crucial to understand the skein relations of the $\GHZ$ tensor and $R$ in $\mathscr{P}_\bullet$ first.
\begin{proposition}\label{pro: symmetric braiding}
	The transposition $R$ is a symmetric braiding, i.e., it satisfies the following relations:
	\begin{itemize}
		\item Type $\I$ Reidemeister move,
		\begin{equation}\label{equ: type I R move}
		\begin{tikzpicture}
		\draw  (0,1)--(.75,.25) arc [radius=.25*1.414,start angle=225, end angle=495]--(0,0);
		\node at (1.5,.5) [right] {$=$};
		\draw  (2.5,0)--(2.5,1);
		\end{tikzpicture}~.
		\end{equation}
		\item Type $\II$ Reidemeister move,
		\begin{equation}\label{equ: type II R move}
		\begin{tikzpicture}
		\draw (0,1.2) arc [radius=1.2/1.732, start angle=60, end angle=-60];
		\draw (.4,1.2) arc [radius=1.2/1.732, start angle=120, end angle=240];
		\node at (.6,.6) [right] {$=$};
		\draw (1.4,0)--(1.4,1.2);
		\draw (1.9,0)--(1.9,1.2);
		\end{tikzpicture}~.
		\end{equation}		
		\item Type $\III$ Reidemeister move,
		\begin{equation}\label{equ: type III R move}
		\begin{tikzpicture}
		\draw (0,0)--(1,1);
		\draw (0,1)--(1,0);
		\draw (.35,0)--(.35,1);
		\node at (1.2,.5) [right] {$=$};
		\draw (2,0)--(3,1);
		\draw (2,1)--(3,0);
		\draw (2.65,0)--(2.65,1);
		\end{tikzpicture}~.
		\end{equation}
     \item Suppose $x\in\mathscr{P}_{m+n}$ for $m,n\geq0$.
     \begin{equation}\label{equ: flatness}
\begin{tikzpicture}
\draw [ultra thick](.5,1.5)--(.5,-.5);
\node at (.45,1.4) [right] {$m$};
\node at (.45,-.4) [right] {$n$};
\draw (-.4,1)--(1.4,1);
\draw [fill=white] (0,0) rectangle (1,.6);
\node at (.5,.3) {$x$};
\node at (.1,.3) [left] {$\$$};
\node at (1.5,.5) [right]{$=$};
\draw [ultra thick](3,-.5)--(3,1.5);
\draw [fill=white] (2.5,1) rectangle (3.5,.4);
\node at (2.95,1.4)[right] {$m$};
\node at (2.95,-.4) [right] {$n$};
\draw(2.1,0)--(3.9,0);
\node at (3,.7) {$x$};
\node at (2.6,.7) [left] {$\$$};
\end{tikzpicture}.
\end{equation}
	\end{itemize}
\end{proposition}
\begin{proof}
	It is straightforward to verify Equation \eqref{equ: type I R move}, \eqref{equ: type II R move}, \eqref{equ: type III R move} and \eqref{equ: flatness} from the definition of the transposition $R$.
\end{proof}
The transposition $R$ plays a crucial role the in study of spin models. The permutation \eqref{equ: permutation} is realized by this rank-$4$ tensor and therefore, this induces an action of the symmetric group $S_n$ acts on $\mathscr{P}_{n}$ for every $n\in\mathbb{N}$. 
\begin{proposition}\label{pro: symmetric group action}
	For $k\in\mathbb{N}$, let $S_{k,n}$ be the element in $\Hom(n,n)$ as follows,
	\begin{figure}[H]
		\begin{tikzpicture}
			\draw (8+.7,0)--(8+1.1,-.5);
		\draw (8+1.1,0)--(8+.7,-.5);
		\draw [ultra thick](8+.35,0)--(8+.35,-.5);
		\draw [ultra thick] (8+1.45,0)--(8+1.45,-.5);
		\node at (8+.4,-.25)[left] {$k-1$};
		\node at (8+1.4,-.25) [right] {$n-k-1$};
		\end{tikzpicture}.
	\end{figure}
Then the map $\pi$ from $S_n$ to $\Hom(n,n)$ by sending the the permutation $(k,k+1)$ to $S_{k,n}$ induces a unitary representation of $S_k$ on the space $\mathscr{P}_{n}$. 
\end{proposition}
\begin{proof}
	By definition, the space $\Hom(n,n)$ consists of all the linear transforms on $\mathscr{P}_n$. For $(i_1,i_2,\cdots,i_n)\in \mathscr{P}_n$,
	\begin{equation}
	S_{k,n}((i_1,i_2,\cdots,i_n))=(i_1,i_2,\cdots,i_{k-1},i_{k+1},i_{k},i_{k+2},\cdots,i_n).
	\end{equation}
	It follows directly that $S_{k,n}$ defines a unitary on $\mathscr{P}_{n}$ and furthermore the map $\pi$ is a unitary representation of $S_n$ on $\mathscr{P}_{n}$. 
\end{proof}
Since the symmetric group $S_n$ is generated by transpositions, $\pi(g)$ can be represented diagrammatically for every $g\in S_n$ and we denote it as follows,
	\begin{figure}[H]
		\begin{tikzpicture}
		\draw [ultra thick] (.5,1.1)--(.5,-.5);
		\draw [fill=white] (0,0) rectangle (1,.6);
		\node at (.5,.3) {$g$};
		\node at (.45,-.25)[right] {$n$};
		\node at (.45,.85)[right] {$n$};
		\end{tikzpicture}.
	\end{figure}
Now we discuss the relations for the $\GHZ$ tensor. First we introduce a notation for the contraction of $\GHZ$ tensor with two of its three indices as follows,
	\begin{equation}\label{equ: single cap}
	\begin{tikzpicture}
	\draw [thick] (0,0)--(0,.5);
	\draw [fill=black] (0,.5) circle [radius=.05];
	\node at (.75,.25) {$=$};
	\draw [thick](1.5,0)--(1.5,.5);
	\draw [thick](1.5,.5)--(1.25,.75) arc [radius=.25*1.414, start angle=225, end angle=-45]--(1.5,.5);
	\draw [fill=black] (1.5,.5) circle [radius=.05];
	\end{tikzpicture}
	\end{equation}
\begin{proposition}\label{pro: relation for GHZ}
	We have the following relations for the $\GHZ$ tensor,
	\begin{equation}\label{equ: HI relation}
	\begin{tikzpicture}[scale=.7]
	\draw [thick] (0,0)--(.5,-.5)--(1,0);
	\draw [thick] (.5,-.5)--(1,-1)--(2,0);
	\draw [thick] (1,-1)--(1,-1.5);
	\node at (2.5,-.75) {$=$};
	\draw [thick](3,0)--(4,-1)--(5,0);
	\draw [thick] (4,0)--(4.5,-.5);
	\draw [thick] (4,-1.5)--(4,-1);
	\draw [fill=black] (.5,-.5) circle [radius=.05];
	\draw [fill=black] (1,-1) circle [radius=.05];
	\draw [fill=black] (4.5,-.5) circle [radius=.05];
	\draw [fill=black] (4,-1) circle [radius=.05];
	\end{tikzpicture}
	\end{equation}
	\begin{equation}\label{equ: normalized}
	\begin{tikzpicture}
	\draw [thick] (0,-.6)--(0,.6);
	\draw [thick,fill=white](0,0) circle [radius=.3];
	\node at (1,0) {$=$};
	\draw [thick](1.7,-.6)--(1.7,.6);
	\draw [fill=black] (0,.3) circle [radius=.05];
	\draw [fill=black] (0,-.3) circle [radius=.05];
	\end{tikzpicture}
	\end{equation}
\end{proposition}
\begin{proof}
		It is straightforward to verify Equation \eqref{equ: HI relation} and \eqref{equ: normalized} from the definition of $\GHZ$ tensor.
\end{proof}
\begin{corollary}\label{cor: GHZ}
	We have the following relation:
	\begin{equation}\label{equ: unit}
	\begin{tikzpicture}
	\draw [thick] (-.5,.5)--(0,0)--(.5,.5);
	\draw [thick,fill=black] (0,0)--(0,-.5) circle [radius=.05];
	\node at (1,0) {$=$};
	\draw [thick](1.5,.25) arc[radius=.5, start angle=180, end angle=360];
	\draw [fill=black] (0,0) circle [radius=.05];
	\end{tikzpicture}
	\end{equation}
\end{corollary}
\begin{proof}
	By Equation \eqref{equ: single cap}, Equation \eqref{equ: HI relation} and \eqref{equ: normalized} we have that
	\begin{equation*}
	\begin{tikzpicture}
	\draw [thick](-.5,.5)--(0,0)--(.5,.5);
	\draw [fill=black] (0,0) circle [radius=.05];
	\draw [thick, fill=black] (0,0)--(0,-.5) circle [radius=.05];
	\node at (1,0) {$=$};
	\draw [thick](2,0)--(1.75,-.25) arc [ radius=.25*1.414, start angle=135, end angle=395]--(2,0)--(2,.25)--(1.5,.75)--(2,.25)--(2.5,.75);
	\draw [fill=black] (2,0) circle [radius=.05];
	\draw [fill=black] (2,.25) circle [radius=.05];
	\node at (3,0) {$=$};
	\draw [thick] (4,0)--(3.75,-.25) arc[radius=.25, start angle=135, end angle=270]--(4.75-.25*.707,-.25*2.707) arc [radius=.25, start angle=270, end angle=395]--(4.5,0)--(4,0)--(3.5,.5);
	\draw [fill=black] (4.5,0) circle [radius=.05];
	\draw [fill=black] (4,0) circle [radius=.05];
	\draw [thick] (4.5,0)--(5,.5);
	\node at (5.5,0) {$=$};
	\draw [thick] (6,.25) arc [radius=.5, start angle=180, end angle=360];
	\end{tikzpicture}
	\end{equation*}
\end{proof}

For a nontrivial subgroup $G\leq S_d$, we have that $\mathscr{P}_\bullet^{S_d}$ sits in $\mathscr{P}_\bullet^G$ as a proper submodel. It turns out that there is exactly one more generator for the group-action model $\mathscr{P}_\bullet^G$ other than $\GHZ$ and the transposition $R$. 
\begin{theorem}\label{thm: generator}
	Let $V$ be a vector space of dimension $d$ and $G$ be a subgroup of the symmetric group $S_d$. Suppose $\mathscr{P}_\bullet^G$ is the group-action model as in Definition \ref{def: groupaction models}. Then $\mathscr{P}_\bullet^G$ is generated by
	\begin{itemize}
		\item $\GHZ$ tensor.
		\item The transposition $R$.
		\item The molecule $S$ defined as follows,
		\begin{equation}\label{equ: full orbit}
		S=[1,2,\cdots, d].
		\end{equation}
	\end{itemize}
\end{theorem}
\begin{proof}
	Let $\mathscr{Q}_\bullet$ be the submodel generated by $\GHZ$ tensor, the transposition $R$ and the molecule $S$. Suppose $\vec{i}=(i_1,i_2,\cdots,i_n)$ is an arbitrary rank-$n$ tensor in $\mathscr{P}_\bullet$. We need to show that $[\vec{i}]\in\mathscr{Q}_n$.
	
	Note that there exists a element $g\in S_n$ such that $g( i_1)\leq g( i_2)\leq\cdots g( i_n)$. It follows directly that 
	\begin{equation}
	\begin{tikzpicture}
	\draw (0,0) rectangle (1,.6);
	\node at (.1,.3) [left] {$\$$};
	\node at (.5,.3) {$[g(\vec{i})]$};
	\draw [ultra thick] (.5,0)--(.5,-1.6);
	\draw [fill=white] (0,-.5) rectangle (1,-1.1);
	\node at (.45,-.25) [right] {$n$};
	\node at (.45,-1.35) [right] {$n$};
	\node at (.5,-.8) {$g^{-1}$};
	\node at (1.5,-.5) {$=$};
	\draw (2,-.5) rectangle (3,.1);
	\node at (2.1,-.2) [left] {$\$$};
	\node at (2.5,-.2) {$[\vec{i}]$};
	\draw [ultra thick] (2.5,-.5)--(2.5,-1);
	\node at (2.45,-.75) [right]{$n$};
	\end{tikzpicture}
	\end{equation}
	Therefore, if $[g(\vec{i})]\in\mathscr{Q}_n$ then $[\vec{i}]\in\mathscr{Q}_n$. This means that we only need to consider the case in which $i_1\leq i_2\leq \cdots\leq i_n$.
	
	Now suppose there exists $j\in\mathbb{N}$ such that $i_j=i_{j+1}$. Let $\vec{i'}$ be the rank-$(n-1)$ tensor $(i_1,i_2,\cdots, i_{j-1},i_j,i_{j+2},\cdots,i_n)$. By definition, we have that
	\begin{equation}
	\begin{tikzpicture}
	\draw (0,0) rectangle (1,.6);
	\node at (.1,.3) [left] {$\$$};
	\node at (.5,.3) {$[\vec{i'}]$};
	\draw (.5,0)--(.5,-.25)--(.6,-.5);
	\draw (.5,-.25)--(.4,-.5);
	\draw [ultra thick] (.2,0)--(.2,-.5);
	\draw [ultra thick] (.8,0)--(.8,-.5);
	\node at (.2,-.25) [left] {\tiny{j-1}};
	\node at (.8,-.25) [right] {\tiny{n-j-1}};
	\node at (2,0) {$=$};
	\draw (2.5,0) rectangle (3.5,.6);
	\node at (2.6,.3)[left] {$\$$};
	\node at (3,.3) {$[\vec{i}]$};
	\draw [ultra thick] (3,0)--(3,-.5);
	\node at (2.95,-.25)[right] {$n$};
	\end{tikzpicture}
	\end{equation}
	Therefore, if $[\vec{i'}]\in\mathscr{Q}_{n-1}$ then $[\vec{i}]\in\mathscr{Q}_n$. This implies that we only need to consider the case in which $i_j's$ must be distinct, namely, $i_1<i_2<\cdots<i_n$. Since $\dim V=d$, this automatically implies that $n<d$. 
	
	Let $T_{\vec{i}}\in\Hom(d,n)$ be defined as $f_1\otimes f_2\otimes \cdots\otimes f_d$ where $f_k=\raisebox{-.2cm}{\begin{tikzpicture}
	\draw [thick] (0,0)--(0,.5);
	\path[fill=white] (-.25,.25)--(.25,.25);
	\end{tikzpicture}}$, if $k\not\in\{i_1,i_2,\cdots,i_n\}$ and $f_k=\raisebox{-.2cm}{\begin{tikzpicture}
	\draw [thick, fill=black] (0,.5)--(0,0) circle [radius=.05];
	\path[fill=white] (-.25,.25)--(.25,.25);
	\end{tikzpicture}}$, if $k\in\{i_1,i_2,\cdots,i_n\}$. We define $S_{\vec{i}}\in\mathscr{Q}_{n}$ as follows,
    \begin{equation}
    \begin{tikzpicture}
    \draw (0,0) rectangle (1,.6);
    \node at (.1,.3) [left] {$\$$};
    \node at (.5,.3) {$S$};
    \draw [ultra thick] (.5,.0)--(.5,-1.6);
    \draw [fill=white] (0,-.5) rectangle (1,-1.1);
    \node at (.1,-.8) [left] {$\$$};
    \node at (.5,-.8) {$T_{\vec{i}}$};    
    \node at (.45,-.25) [right] {$d$};
    \node at (.45,-1.35)[right] {$n$};
    \end{tikzpicture}
    \end{equation}
    By definition, we have that
    \begin{equation}
    S_{\vec{i}}=\sum_{g\in G}(g\cdot i_1,g\cdot i_2,\cdots, g\cdot i_n)=[\vec{i}].
    \end{equation}
    Therefore, we have $[\vec{i}]\in\mathscr{Q}_\bullet$ for arbitrary rank-n tensor $\vec{i}$. This implies that the group-action model $\mathscr{P}_\bullet^G$ is generated by $\GHZ$ tensor, the transposition $R$ and the molecule $S$.
 \end{proof}
By Theorem \ref{thm: generator}, we determine the generators for the group-action model $\mathscr{P}^G_\bullet$. Clearly, there are more relations need to determine the structure of the group-action model $\mathscr{P}_\bullet^G$. Before introducing these relations, let us introduce a few notations as follows.
\begin{notation}
	\begin{itemize}
		Suppose $k$ is a positive integer. 
			\item 
	\begin{equation}
	\begin{tikzpicture}
	\draw [ultra thick] (0,.75)--(0,0)--(-.75,-.75)--(0,0)--(.75,-.75);
	\draw [fill=white](0,0) circle [radius=.25];
	\node at (0,0) {$k$};
	\node at (1.5,0) {$=$};
	\draw (4,.75)--(4,0)--(3.25,-.75)--(4,0)--(4.75,-.75);
	\draw (3,.75)--(3,0)--(2.25,-.75)--(3,0)--(3.75,-.75);
	\draw (3.3,.75)--(3.3,0)--(2.55,-.75)--(3.3,0)--(4.05,-.75);
	\draw [thick, dotted] (3.4,.5)--(3.9,.5);
	\draw [dotted, thick] (2.95,-.5)--(3.25,-.5);
	\draw [dotted, thick] (3.95,-.5)--(4.35,-.5);
	\draw[decoration={brace,mirror,raise=5pt},decorate, thick, blue]
	(2.15,-.65) -- node[below=6pt] {$k$} (3.35,-.65);
	\draw[decoration={brace,mirror,raise=5pt},decorate, thick, blue]
	(3.65,-.65) -- node[below=6pt] {$k$} (4.85,-.65);
	\draw[decoration={brace,raise=5pt},decorate, thick, blue]
	(2.9,.65) -- node[above=6pt] {$k$} (4.1,.65);
	\end{tikzpicture}
	\end{equation}
	\item
	 \begin{equation}
	\begin{tikzpicture}
	\draw [ultra thick] (0,.5)--(0,-.5);
	\draw [fill=black] (0,-.5) circle [radius=.05];
	\node at (-.05,0) [right] {$k$};
	\node at (.75,0) {$=$};
	\draw [fill=black,thick] (1.2,.5)--(1.2,-.5) circle [radius=.05];
	\draw [fill=black,thick] (1.5,.5)--(1.5,-.5) circle [radius=.05];
	\draw [fill=black,thick] (2.2,.5)--(2.2,-.5) circle [radius=.05];
	\draw [dotted, thick](1.6,0)--(2.1,0);
	\draw[decoration={brace,raise=5pt},decorate, thick, blue]
	(1.1,.4) -- node[above=6pt] {$k$} (2.3,.4);
	\end{tikzpicture}
	\end{equation}
\end{itemize}
\end{notation}
\begin{theorem}\label{thm: relation for S}
	Suppose $\mathscr{P}_\bullet$ is a spin model associated with a vector space $V$ of dimension $d$ and $G$ is a subgroup of the symmetric group $S_d$. Let $\mathscr{P}_\bullet^G$ be the group-action model. Let $S$ be the molecule, i.e., $S=[1,2,\cdots,n]$. Then we have the following relations hold,
	\begin{itemize}
		\item \textbf{Capped scalar:}
		\begin{equation}\label{equ: all balck caps}
		\begin{tikzpicture}
		\draw (0,0) rectangle (1,.6);
		\draw [fill=black, ultra thick] (.5,0)--(.5,-.5);
		\node at (.45,-.25) [right] {$d$};
		\node at (.1,.3) [left] {$\$$};
		\node at (.5,.3) {$S$};
		\draw [fill=black] (.5,-.5) circle [radius=.05];
		\node at (1.2,0)[right] {$=\vert G\vert$};
		\end{tikzpicture}
		\end{equation}
		\item \textbf{Y-uncappable:} For every $\sigma\in S_d$, we have
		\begin{equation}\label{equ: Y uncappable}
		\begin{tikzpicture}
		\draw (0,0) rectangle (1,.6);
		\node at (.1,.3) [left] {$\$$};
		\node at (.5,.3) {$S$};
		\draw [ultra thick] (.5,0)--(.5,-.5);
		\node at (.45,-.25) [right] {$d$};
		\draw (0,-.5) rectangle (1,-1.1);
		\node at (.5,-.8) {$\sigma$};
		\draw [thick](.1,-1.1)--(.2,-1.35)--(.3,-1.1)--(.2,-1.35)--(.2,-1.6);
		\draw [ultra thick] (.6,-1.1)--(.6,-1.6);
		\node at (.55,-1.35) [right] {$d-2$};
		\node at (1.8,-.5) {$=0$};
		\end{tikzpicture}
		\end{equation}	
		\item  \textbf{Group symmetrizing:} Let $p=\sum_{g\in G} \pi(g)\in\Hom(d,d)$, we have that
		\begin{equation}\label{equ: reduction}
		\begin{tikzpicture}
		\draw (0,0) rectangle (1,.6);
		\draw (1.5,0) rectangle (2.5,.6);
		\node at (.1,.3) [left] {$\$$};
		\node at (1.6,.3) [left] {$\$$};
		\node at (.5,.3) {$S$};
		\node at (2,.3) {$S$};
		\draw [ultra thick] (.5,0)--(.5,-.5);
		\draw [ultra thick] (2,0)--(2,-.5);
		\node at (.6,-.25)[left] {$d$};
		\node at (2.1,-.25) [left] {$d$};
		\node at (2.5,0)[right] {$=\displaystyle\sum_{g\in G}$};
		\draw (4,0) rectangle (5,.6);
		\draw [ultra thick](4.5,0)--(4.5,-.5)--(4,-1) arc [radius=.05, start angle=135, end angle=180]--(4+.0707-.1,-2);
		\draw [ultra thick](4.5,-.5)--(5,-1) arc[radius=.05, start angle=45, end angle=0]--(5+.1-.0707,-2);
		\draw  [fill=white] (4.5,-.5) circle [radius=.2];
		\draw [fill=white] (4.5,-1.2) rectangle (5.5,-1.7);
		\node at (4.1,.3) [left] {$\$$};
		\node at (4.5,.3) {$S$};
		\node at (5,-1.5) {$g$};
		\node at (4.5,-.5) {$d$};
		\node at (6,0) {$=$};
		\draw (6.5,0) rectangle (7.5,.6);
		\node at (6.6,.3)[left] {$\$$};
		\node at (7,.3) {$S$};
		\draw [ultra thick](2.5+4.5,0)--(2.5+4.5,-.5)--(2.5+4,-1) arc [radius=.05, start angle=135, end angle=180]--(2.5+4+.0707-.1,-2);
		\draw [ultra thick](2.5+4.5,-.5)--(2.5+5,-1) arc[radius=.05, start angle=45, end angle=0]--(2.5+5+.1-.0707,-2);
		\draw [fill=white] (7,-.5) circle [radius=.2];
		\draw [fill=white] (7,-1.2) rectangle (8,-1.7);
		\node at (7,-.5) {$d$};
		\node at (7.5,-1.5) {$p$};
		\end{tikzpicture}
		\end{equation}
	\end{itemize}
\end{theorem}
\begin{proof}
	By definition, we have that 
	\begin{equation}
	S=\sum_{g\in G}(g\cdot 1,g\cdot 2,\cdots, g\cdot D).
	\end{equation}
	Since $g$ is a permutation on the set $\{1,2,\cdots, d\}$, $g\cdot i\neq g\cdot j$ if $i\neq j$. This implies that $S$ is a linear combination of rank-$D$ tensors whose components are distinct. Therefore, Equation \eqref{equ: Y uncappable} follows directly from the definition of $\GHZ$ tensor. Similarly, Equation \eqref{equ: all balck caps} follows from a direct computation.
	
	For Equation \eqref{equ: reduction}, by definition the left-hand side is the tensor product of two $S$'s. Therefore, we have that
	\begin{align*}
	S\otimes S&=\sum_{g,h\in G} (g\cdot 1,g\cdot 2,\cdots, g\cdot d)\otimes (h\cdot 1,h\cdot 2,\cdots, h\cdot d)\\
	&=\sum_{g,h\in G} (g\cdot 1,g\cdot 2,\cdots, g\cdot d, hg^{-1}g\cdot 1,hg^{-1}g\cdot 2,\cdots, hg^{-1}g\cdot d)\\
	&=\sum_{h'\in G}\sum_{g\in G}(g\cdot 1,g\cdot 2,\cdots, g\cdot d, h'g\cdot 1,h'g\cdot 2,\cdots, h'g\cdot d)
	\end{align*}
	This ends up with the same as the right-hand side of Equation \eqref{equ: reduction}. Therefore, Equation \eqref{equ: reduction} holds. 
\end{proof}
The relations in Theorem \ref{thm: relation for S} combined with Equation \eqref{equ: circle parameter}, \eqref{equ: HI relation}, \eqref{equ: normalized}, \eqref{equ: type I R move}, \eqref{equ: type II R move}, \eqref{equ: type III R move}, \eqref{equ: flatness} are sufficient to provide an evaluation algorithm for $\mathscr{P}^G_\bullet$ which will be proved in Theorem \ref{thm: skein theory}. Before that, we discuss the special case when $G=S_d$, since $\mathscr{P}^{S_d}_\bullet$ to group-action models is what Temperley-Lieb-Jones planar algebras to planar algebras. 
\begin{definition}\label{def: planar}
	Suppose $T$ is a diagram in $\Hom(n,m)$. Then $T$ is said to be \underline{planar} if $T$ is generated by $\GHZ$ tensor.
\end{definition}
\begin{remark}
	If a diagram $T\in\Hom(n,m)$ is generated by $\GHZ$ tensor, then $T$ is literally a planar graph with $n+m$ open edges. This justifies the word \underline{planar} in Definition \ref{def: planar}. Equation \eqref{equ: normalized} says that the graph must be simply laced. We say $T$ is a standard form if $T$ is a simply-laced planar graph. It follows directly that every element in $\Hom(n,m)$ generated by $\GHZ$ tensor is a linear combination of standard forms. 
\end{remark}
\begin{lemma}[Jones, \cite{Jon94}]\label{lem: planarify}
	Let $\mathscr{P}^{S_d}_\bullet$ be the group-action model with respect to the symmetric group $S_d$ and $T\in \mathscr{P}_{n,n}\subset\Hom(n,m)$ for some $n\in\mathbb{N}$. Then there exists two permutations $\alpha\in S_n,\beta\in S_m$ such that $\pi(\beta)\circ T\circ\pi(\alpha)$ is planar. 
\end{lemma}
\begin{theorem}\label{thm: skein theory}
	Suppose $V$ is a vector space of dimension $d$ and $G$ is a subgroup of the symmetric group $S_d$. Let $\mathscr{P}^G_\bullet$ be the group-action model. Then $\mathscr{P}^G_\bullet$ is generated by $\GHZ$ tensor, the transposition $R$ and the molecule $S$. They satisfies Equation \eqref{equ: circle parameter}, \eqref{equ: HI relation}, \eqref{equ: normalized}, \eqref{equ: type I R move}, \eqref{equ: type II R move}, \eqref{equ: type III R move}, \eqref{equ: flatness}, \eqref{equ: Y uncappable}, \eqref{equ: all balck caps} and \eqref{equ: reduction}. Furthermore, there is an evaluation algorithm for the group-action model $\mathscr{P}_\bullet^G$ designed with respect to above generators and relations. 
\end{theorem}
\begin{proof}
	By Theorem \ref{thm: generator}, we know that $\mathscr{P}_\bullet^G$ is generated by $\GHZ$ tensor, the transposition $R$ and the molecule $S$. All relations are proved in previous propositions and theorems. Therefore, we need to provide an explicit evaluation algorithm for $\mathscr{P}^G_\bullet$. 	
	\begin{enumerate}
		\item Suppose $X$ is a diagram in $\mathscr{P}_0$ without $S$. Then $X$ is essentially a cubic graph and we define the evaluation of $X$ to be the scalar $d^{C_X}$, where $C_X$ is the number of connected components in the graph $X$.  
		\item Suppose $X$ is a diagram in $\mathscr{P}_0$ whose number of the generators $S$ is at least one. By isotopy, we put $X$ in $[0,1]\times[0,1]\subset\mathbb{R}^2$ such that there is no overlapping for the $x$-coordinates of any two generators. (Note that the molecule $S$ is a rectangle and the other two generators are single points. ) In this form, there are either nothing or parallel strings above a generator $S$. By Equation \eqref{equ: flatness}, a generator $S$ can always be moved up vertically to be above the line $y=1$. Therefore, $X$ can be written in the following form:
	
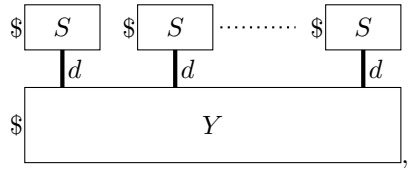
\begin{figure}[H]
		\begin{tikzpicture}
		\draw (0,0) rectangle (1,.6);
		\node at (.1,.3) [left] {$\$$};
		\node at (.5,.3) {$S$};
		\draw (1.5,0) rectangle (2.5,.6);
		\node at (1.6,.3) [left] {$\$$};
		\node at (2,.3) {$S$};
		\draw (4,0) rectangle (5,.6);
		\node at (4.1,.3) [left] {$\$$};
		\node at (4.5,.3) {$S$};
		\draw [dotted, thick](2.6,.3)--(3.6,.3);
		\draw [ultra thick] (.5,-.5)--(.5,0);
		\draw [ultra thick] (2,-.5)--(2,0);
		\draw [ultra thick] (4.5,0)--(4.5,-.5);
		\node at (.45,-.25) [right] {$d$};
		\node at (1.95,-.25) [right] {$d$};
		\node at (4.45,-.25) [right] {$d$};
		\draw (0,-.5) rectangle (5,-1.5);
		\node at (2.5,-1) {$Y$};
		\node at (.1,-1)[left] {$\$$};
		\end{tikzpicture},
	\end{figure}
    where $Y$ is a diagram generated by $\GHZ$ tensor and the transposition $R$. By Equation \eqref{equ: reduction}, any two generators $S$ can be reduced to a linear combination of diagrams whose number of $S$ is one fewer. Therefore, by induction we only need to evaluate diagrams whose number of generator $S$ is exactly one, namely, we can assume $X$ is the following form,
    
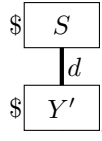
\begin{figure}[H]
    	\begin{tikzpicture}
    	\draw (0,0) rectangle (1,.6);
    	\node at (.1,.3) [left] {$\$$};
    	\node at (.5,.3) {$S$};
    	\draw [ultra thick] (.5,-1.1)--(.5,0);
    	\node at (.45,-.25) [right] {$d$};
    	\draw [fill=white] (0,-.5) rectangle (1,-1.1);	
    	\node at (.1,-.8) [left] {$\$$};
    	\node at (.5,-.8) {$Y'$};
    	\end{tikzpicture},
    \end{figure}
    where $Y'$ is a diagram generated by $\GHZ$ tensor and the transposition $R$. By Lemma \ref{lem: planarify}, we know that there exists a permutation $\sigma\in S_d$ such that $\pi(\sigma) Y'$ is planar. Let $Y''=\pi(\sigma)Y'$ and thus $X$ equals to the following form,
    
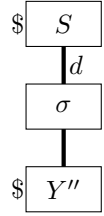
\begin{figure}[H]
    	\begin{tikzpicture}
    	\draw (0,0) rectangle (1,.6);
    	\node at (.1,.3) [left] {$\$$};
    	\node at (.5,.3) {$S$};
    	\draw [ultra thick] (.5,-1.6)--(.5,0);
    	\node at (.45,-.25) [right] {$d$};
    	\draw [fill=white] (0,-.5) rectangle (1,-1.1);	
    	\node at (.1,-1.9) [left] {$\$$};
    	\node at (.5,-.8) {$\sigma$};
    	\draw (0,-1.6) rectangle (1,-2.2);
    	\node at (.5,-1.9) {$Y''$};
    	\end{tikzpicture}
    \end{figure}
    If $Y''=\raisebox{-.2cm}{\begin{tikzpicture}
    	\draw [white] (-.25,0)--(0,0);
    	\draw [ultra thick, fill=black] (0,.5)--(0,0) circle [radius=.05];
    	\node at (-.05,.25) [right] {$d$};
    	\end{tikzpicture}}$, then $X=\vert G\vert$ by Equation \eqref{equ: all balck caps}; Otherwise $X=0$ by Equation \eqref{equ: Y uncappable}. 
\end{enumerate}
Next, for every $n\in\mathbb{N}$, we give explicitly the standard forms for $\mathscr{P}_n^G$, such that every element in $\mathscr{P}^G_n$ can be written as a linear combination of standard forms. A diagram $Y$ is said to be a standard form if either there exists an diagram $T\in\mathscr{P}_n$ which is planar and a permutation $\sigma\in S_n$ such that $Y=\pi(\sigma)\circ T$ or there exists an diagram $T\in\Hom(d,n)$ which is planar and $\alpha\in S_d, \beta\in S_n$ such that $Y=\pi(\beta)\circ T\circ\pi(\alpha) S$. It follows that there are only finitely many standard forms for $\mathscr{P}_n^G$. Suppose $X$ is an arbitrary diagram in $\mathscr{P}_n^G$. We have two different cases as before
\begin{enumerate}
	\item Suppose the number of $S$ in $X$ is zero. By Lemma \ref{lem: planarify}, we know that $X$ must be one of the standard forms. 
	\item Suppose the number of $S$ in $X$ is at least one. One can rewrite $X$ into the following form as before,
	
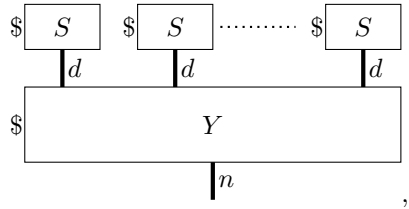
\begin{figure}[H]
			\begin{tikzpicture}
		\draw (0,0) rectangle (1,.6);
		\node at (.1,.3) [left] {$\$$};
		\node at (.5,.3) {$S$};
		\draw (1.5,0) rectangle (2.5,.6);
		\node at (1.6,.3) [left] {$\$$};
		\node at (2,.3) {$S$};
		\draw (4,0) rectangle (5,.6);
		\node at (4.1,.3) [left] {$\$$};
		\node at (4.5,.3) {$S$};
		\draw [dotted, thick](2.6,.3)--(3.6,.3);
		\draw [ultra thick] (.5,-.5)--(.5,0);
		\draw [ultra thick] (2,-.5)--(2,0);
		\draw [ultra thick] (4.5,0)--(4.5,-.5);
		\node at (.45,-.25) [right] {$d$};
		\node at (1.95,-.25) [right] {$d$};
		\node at (4.45,-.25) [right] {$d$};
		\draw [ultra thick](2.5,-1.5)--(2.5,-2);
		\node at (2.45,-1.75) [right] {$n$};
		\draw (0,-.5) rectangle (5,-1.5);
		\node at (2.5,-1) {$Y$};
		\node at (.1,-1)[left] {$\$$};
		\end{tikzpicture},
	\end{figure}
where $Y$ is a diagram generated by $\GHZ$ tensor and the transposition $R$. By Equation \eqref{equ: reduction}, any two generators $S$ can be reduced to a linear combination of diagrams whose number of the generator $S$ is one fewer. Therefore, $X$ can be written as a linear combination of diagrams in the following form,
\begin{figure}[H]
	\begin{tikzpicture}
	\draw (0,0) rectangle (1,.6);
	\node at (.1,.3) [left] {$\$$};
	\node at (.5,.3) {$S$};
	\draw [ultra thick] (.5,-1.6)--(.5,0);
	\node at (.45,-.25) [right] {$d$};
	\draw [fill=white] (0,-.5) rectangle (1,-1.1);	
	\node at (.1,-.8) [left] {$\$$};
	\node at (.5,-.8) {$Y'$};
	\node at (.45,-1.35)[right] {$n$};
	\end{tikzpicture},
\end{figure}
where $Y'$ is a diagram generated by $\GHZ$ tensor and the transposition $R$. By Lemma \ref{lem: planarify}, we know that each diagram in the above form is a standard form. 
\end{enumerate}
Therefore, the above evaluation algorithm combined with the generators and relations provides a complete skein theory for the group-action model $\mathscr{P}^G_\bullet$.
\end{proof}
\begin{corollary}\label{cor: group-subgroup}
	Suppose $G$ is a finite group and $H$ is a subgroup of $G$. Then we have the same skein theory for the subfactor planar algebra for the group-subgroup subfactor $R\rtimes H\subset R\rtimes G$. 
\end{corollary}
\begin{proof}
	We set $d$ to be the index of $H$ in $G$, $[G:H]$. Consider the set $X=\{Hg_1,Hg_2,\cdots,Hg_d\}$, the set of right cosets of $H$. Then $G$ acts on $X$ transitively by left multiplication and the stabilizer of each point in $X$ is $H$. Therefore, one can construct a group action model $\mathscr{P}^G_\bullet$ based on this group action and in particular $\mathscr{P}_\bullet^G$ is the even part of the subfactor planar algebra for the group-subgroup subfactor $R\rtimes H\subset R\rtimes G$. Therefore the skein theory provided in Theorem \ref{thm: skein theory} completely the even part of the group-subgroup subfactor planar algebra for $H\subset G$, denoted by $\mathscr{P}_{\bullet,\pm}$. In this subfactor planar algebra, we have that for $x\in\mathscr{P}_{1,+}$, 
	\begin{equation}
	\begin{tikzpicture}
	\draw [fill=lightgray] (.5,-.2)--(.5,.8) arc [radius=.4, start angle=180, end angle=0]--(1.3,-.2) arc [radius=.4, start angle=0,end angle=-180];
	\draw [fill=white] (0,0) rectangle (1,.6);
	\node at (.5,.3) {$x$};
	\node at (.1,.3) [left] {$\$$};
	\node at (1.65,.3) {$=d$};
	\path [fill=lightgray] (2,-.8) rectangle (4,1.4); 
	\draw [fill=white] (2.2,-.2)--(2.2,.8) arc [radius=.4, start angle=180, end angle=0]--(3,-.2) arc [radius=.4, start angle=0, end angle=-180];
	\draw [fill=white] (2.6,0) rectangle (3.6,.6);
	\node at (3.1,.3) {$x$};
	\node at (2.7,.3) [left] {$\$$};
	\end{tikzpicture}
	\end{equation}
	Therefore, the skein theory completely describes the whole planar algebra $\mathscr{P}_{\bullet,\pm}$. 
\end{proof}
\section{Example}\label{sec: Example}
In this section, we discuss the example in Question \ref{ques: equiv formuation of ques of VJ}, namely the group-action model associated with the Petersen graph. Let us first recall the definition of the Petersen graph.
\begin{definition}\label{def: Petersen graph}
	The Petersen graph $\Gamma$ is given as follows,
	
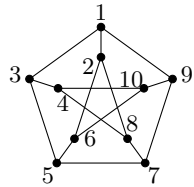
\begin{figure}[H]
		\begin{tikzpicture}
		\foreach \i in {1,...,5}
		{
			\pgfmathsetmacro{\a}{cos(360*(\i-1)/5+90)};
			\pgfmathsetmacro{\b}{sin(360*(\i-1)/5+90)};
			\pgfmathsetmacro{\r}{cos(360*(\i)/5+90)};
			\pgfmathsetmacro{\s}{sin(360*(\i)/5+90)};
			\pgfmathsetmacro{\p}{cos(360*(\i+1)/5+90)};
			\pgfmathsetmacro{\q}{sin(360*(\i+1)/5+90)};
			\pgfmathsetmacro{\k}{int(2*\i-1)};
			\pgfmathsetmacro{\l}{int(2*\i)};
			\node at (1.2*\a,1.2*\b) {\small{$\k$}};
			\node at (.42*\a+.18*\r,.42*\b+.18*\s) {\small{$\l$}};
			\draw (\a,\b)--(\r,\s);
			\draw (\a,\b)--(.6*\a,.6*\b);
			\draw (.6*\a,.6*\b)--(.6*\p,.6*\q);
			\draw [fill=black] (\a,\b) circle [radius=.05];
			\draw [fill=black] (.6*\a,.6*\b) circle [radius=.05];
		}
		\end{tikzpicture}
	\end{figure}
The Petersen graph can also be constructed as the Kneser graph $KG_{5,2}\colon$
\begin{enumerate}
	\item The vertex set $V$ consists of all 2-subsets of $\{1,2,3,4,5\}$, i.e., $\{v\subset\{1,2,3,4,5\}: \vert v\vert =2\}$.
	\item The edge set $E$ consists of all unordered pairs of two vertices who correspond to disjoint 2-subsets, i.e., for $v_1,v_2\in V$, 
	\begin{equation*}
	(v_1,v_2)\in E\Leftrightarrow v_1\cap v_2=\emptyset.
	\end{equation*}
\end{enumerate}
Therefore, it follows that the automorphism group of the Petersen graph is the symmetric group $S_5$ and the action of $S_5$ on the vertex set $V$ is transitive. Furthermore, the stabilizer group of a single vertex is isomorphic to $S_2\times S_3$. 
\end{definition}
Let $W$ be the vector space freely generated by the vertex set $V$ of the Petersen graph $\Gamma$. Therefore, the symmetric group $S_5$ acts on $W$ by permuting its basis. By Definition \ref{def: groupaction models}, we denote the above group-action model associated with the Petersen graph by $\mathscr{P}_\bullet^{S_5}$. 
\begin{proposition}\label{pro: basis for 2-box}
	The space $\mathscr{P}_{2}^{S_5}$ has a basis $\{1,A_\Gamma,A_{\Gamma^C}\}$, where $1$ is the identity matrix and $A_\Gamma$ (resp. $A_{\Gamma^c}$) is the adjacency matrix for the Petersen graph $\Gamma$ (resp. the complement of $\Gamma$ with respect to the complete graph $K_{10}$).
\end{proposition}
\begin{proof}
	By definition, the space $\mathscr{P}_2^{S_5}$ has a basis which consists of the orbits of basics tensors under the action of $S_5$, namely,
	\begin{itemize}
		\item $[(1,2),(1,2)]$, 
		\item $[(1,2),(3,4)]$,
		\item $[(1,2),(1,3)]$. 
	\end{itemize}
When considering $\mathscr{P}_2^{S_5}$ as $\Hom(W,W)$, the above three elements are exactly $\{1,A_\Gamma,A_{\Gamma^C}\}$. 
\end{proof}
\begin{notation}
	The basis for $\mathscr{P}_2^{S_5}$ in Proposition \ref{pro: basis for 2-box} is represented as follows,
		
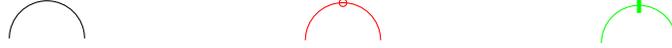
\begin{figure}[H]
		\begin{minipage}{.25\textwidth}
			\centering
			\begin{tikzpicture}
			\draw [black] (0,0) arc [radius=.5, start angle=180, end angle=0];
			\end{tikzpicture}
		\end{minipage}
		\begin{minipage}{.25\textwidth}
			\centering
			\begin{tikzpicture}
				\draw [red] (0,0) arc [radius=.5, start angle=180, end angle=0];
				\draw [red] (.5,.5) circle [radius=.05];
			\end{tikzpicture}
		\end{minipage}
	\begin{minipage}{.25\textwidth}
		\centering 
		\begin{tikzpicture}
			\draw [green] (0,0) arc [radius=.5, start angle=180, end angle=0];
			\draw [green, ultra thick] (.5,.4)--(.5,.6);
		\end{tikzpicture}
	\end{minipage}
	\end{figure}
Moreover, we have the following relation,
\begin{equation}\label{equ: 2-box identity decomposition}
\begin{tikzpicture}
\draw [fill=black](0,0)--(0,.3) circle [radius=.05];
\draw [fill=black](1,0)--(1,.3) circle [radius=.05];
\node at (1.5,.25) {$=$};
\draw (2,0) arc [radius=.5, start angle=180, end angle=0];
\node at (3.5,.25) {$+$};
\draw [red] (4,0) arc [radius=.5, start angle=180, end angle=0];
\draw [red] (4.5,.5) circle [radius=.05]; 
\node at (5.5,.25) {$+$};
\draw [green] (6,0) arc [radius=.5, start angle=180, end angle=0];
\draw [green, ultra thick] (6.5,.4)--(6.5,.6);
\end{tikzpicture}
\end{equation}
\end{notation}
The following theorem is known to Jones in his unpublished work \cite{JonPC} and also to Curtin in \cite{Cur03}. In this paper, we give an independent proof for the group-action model associated with the Petersen graph. 
\begin{theorem}\label{thm: 2-box generating}
	Let $\mathscr{P}_\bullet^{S_5}$ be the group-action model associated with the Petersen graph and $\mathscr{Q}_\bullet$ be the submodel generated by $\mathscr{P}_2^{S_5}$ and the $\GHZ$ tensor. Then the following are equivalent
	\begin{enumerate}
		\item \label{itm: Petersen 1} $\mathscr{Q}_\bullet=\mathscr{P}_\bullet^{S_5}$.
		\item \label{itm: Petersen 2} $R\in\mathscr{Q}_{4}$. 
	\end{enumerate}
\end{theorem}
\begin{proof}
	It is obvious that \eqref{itm: Petersen 1}$\Rightarrow$\eqref{itm: Petersen 2}. We only need to show the other direction, i.e., \eqref{itm: Petersen 2}$\Rightarrow$\eqref{itm: Petersen 1}.
	
	By Theorem \ref{thm: generator}, we know that $\mathscr{P}_\bullet^{S_5}$ is generated by the transposition $R$ and the molecule $S$. Suppose \eqref{itm: Petersen 2} holds, i.e., $R\in\mathscr{Q}_{4}$. This implies that \raisebox{-.3cm}{\begin{tikzpicture}
		\draw [red] (0,0)--(.8,.8);		
		\draw (0,.8)--(.8,0);
		\draw [red] (.2,.2) circle [radius=.05];
		\end{tikzpicture}}$\in\mathscr{Q}_4$. Now we give an explicit construction of the molecule $S$ based on the Petersen graph $\Gamma$. For each vertex of $\Gamma$, we put a $\GHZ_4$ tensor as in Definition \ref{def: GHZ}. For each edge of $\Gamma$, we put the rank-$2$ tensor $A_\Gamma$ and then apply contraction on those tensors following the pattern of $\Gamma$ as on the left of Fig. \ref{fig: construction of s}. Then we rearrange the diagram to obtain a rank-$10$ tensor, denoted by $\widetilde{S}$, as illustrated on the right of Fig. \ref{fig: construction of s}.
	\begin{figure}[H]
		\centering
		\begin{minipage}{.4\textwidth}
			\begin{tikzpicture}
			\foreach \i in {1,...,5}
			{
				\pgfmathsetmacro{\a}{1.5*cos(360*(\i-1)/5+90)};
				\pgfmathsetmacro{\b}{1.5*sin(360*(\i-1)/5+90)};
				\pgfmathsetmacro{\r}{1.5*cos(360*(\i)/5+90)};
				\pgfmathsetmacro{\s}{1.5*sin(360*(\i)/5+90)};
				\pgfmathsetmacro{\p}{1.5*cos(360*(\i+1)/5+90)};
				\pgfmathsetmacro{\q}{1.5*sin(360*(\i+1)/5+90)};
				\pgfmathsetmacro{\k}{int(2*\i-1)};
				\pgfmathsetmacro{\l}{int(2*\i)};
				\draw (\a,\b)--(\a,\b+.35);
				\draw (.6*\a-.15,.6*\b)--(.6*\a-.15,.6*\b+.35);
				\draw [red] (\a,\b)--(\r,\s);
				\draw [red] (.5*\a+.5*\r,.5*\b+.5*\s) circle [radius=.05];
				\draw [red](.6*\a-.15,.6*\b)--(.6*\p-.15,.6*\q);
				\draw [red] (.3*\a+.3*\p-.15,.3*\b+.3*\q) circle [radius=.05];
				\draw [red](\a,\b)--(.6*\a-.15,.6*\b);
				\draw [red] (.8*\a-.075,.8*\b) circle [radius=.05];
				\draw [fill=black] (\a,\b) circle [radius=.05];
				\draw [fill=black] (.6*\a-.15,.6*\b) circle [radius=.05];
			}
			\end{tikzpicture}
		\end{minipage}
	\begin{minipage}{.4\textwidth}
			\begin{tikzpicture}
		\foreach \i in {1,...,5}
		{
			\ifthenelse{\NOT \i=3}{\pgfmathsetmacro{\a}{(2*\i+Mod(2*\i+4,10))/2}}{\pgfmathsetmacro{\a}{2*\i+2}};
			\pgfmathsetmacro{\b}{(2*\i+Mod(2*\i+4,10))/2};
			\pgfmathsetmacro{\r}{abs(2*\i-\a)};
			\draw [red] (\a/2-\r/2,1.25) arc [radius=\r/2, start angle=180, end angle=0];
			\draw [red] (\a/2,1.25+\r/2) circle [radius=.05]; 
			\draw [red] (\i-.5,0) arc [radius=.25, start angle=180, end angle=0];
			\draw [red] (\i-.25,.25) circle [radius=.05];
			\ifthenelse{\NOT \i=5}{\draw [red] (\i-.5,.25) arc [radius=.5, start angle=180, end angle=0];\draw [red] (\i-.1,.73) circle [radius=.05];}{\draw [red] (\i-.5,.25)--(\i-.5, .5) arc [radius=.5, start angle=0, end angle=90]--(1,1) arc [radius=.5, start angle=90, end angle=180]--(.5,.25);\draw [red] (2.5,1) circle [radius=.05];};
			\draw (\i,0)--(\i,1.25);
			\draw (\i-.5,0)--(\i-.5,.25);
			\draw [fill=black] (\i,0) circle [radius=.05];
			\draw [fill=black] (\i-.5,0) circle [radius=.05];
			\draw [fill=black] (\i,1.25) circle [radius=.05];
			\draw [fill=black] (\i-.5,.25) circle [radius=.05];
			\draw (\i,0)--(\i,-.5);
			\draw (\i-.5,0)--(\i-.5,-.5);
		}
		\end{tikzpicture}
	\end{minipage}
		\caption{The construction of the molecule $S$.}\label{fig: construction of s}
	\end{figure}
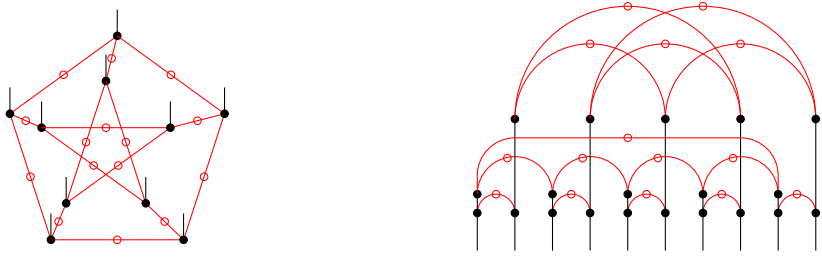	
The rank-$10$ tensor $\tilde{S}$ is obtained by deforming the Petersen graph $\Gamma$. Let $\vec{i}=(i_1,i_2,\cdots,i_{10})$ be an arbitrary rank-$10$ tensor in $\mathscr{P}_{10}$. Then we have
\begin{equation}
\langle \tilde{S},\vec{i}\rangle=\prod_{(j,k)\in E} A_\Gamma(i_j,i_k),
\end{equation}
where $E$ is the set of edges in the Petersen graph $\Gamma$ and $A_\Gamma$ is the adjacency matrix of $\Gamma$. Note that $A_\Gamma$ is a zero-one matrix. The value $\langle \tilde{S},\vec{i}\rangle$ must be either 0 or 1. Let $\vec{i}$ be a rank-$10$ tensor such that $\langle \tilde{S},\vec{i}\rangle=1$. First we show that $(i_1,i_2,\cdots,i_{10})$ is a permutation of $(1,2,\cdots, 10)$. Otherwise, there exists $j\neq k$ such that $i_j=i_k$. Since $\langle \tilde{S},\vec{i}\rangle=1$, we have that $(j,k)\not\in E$. By the definition of the Petersen graph, there exists $1\leq m,n\leq 10$ such that $m\neq n$ and $(j,m),(m,n),(n,k)\in E$. This implies that $(i_j,i_m),(i_m,i_n),(i_n,i_j)\in E$ which leads to a contradiction since the Petersen graph does not contains $K_3$ as a subgraph. Therefore, there exists $g\in S_{10}$ such that $\vec{i}=g(1,2,\cdots,10)$. Moreover, the map $g$ is actually a graph automorphism and thus $g\in Aut(\Gamma)=S_5$. So we have that
\begin{equation}
\tilde{S}=\sum_{g\in S_5} g(1,2,\cdots,10).
\end{equation}
By the definition of $S$, we know that $\tilde{S}=S$ and hence $S\in\mathscr{A}_\bullet$. By Theorem \ref{thm: generator}, we have that $\mathscr{A}_\bullet=\mathscr{P}_{\bullet}^{S_5}$.
\end{proof}
\begin{theorem}\label{thm: condition1}
The group-action model $\mathscr{P}_{\bullet}^{S_5}$ associated with the Petersen graph is generated by $\mathscr{P}_{2}^{S_5}$. 
\end{theorem}
\begin{proof}
	Let $\mathscr{Q}_\bullet$ be submodel generated by $\mathscr{P}_{2}^{S_5}$. By Theorem \ref{thm: 2-box generating}, we only need to show that the transposition $R$ belongs to $\mathscr{Q}_4$. By Equation \eqref{equ: 2-box identity decomposition}, we have that
	\begin{equation}\label{equ: decomposition of R}
	\begin{tikzpicture}
	\draw  (0,0)--(1,1);
	\draw  (0,1)--(1,0);
	\node at (1.5,.5) {$=$};
	\draw (2,0)--(3,1);
	\draw (2,1)--(3,0);
	\draw (2.75,.75)--(2.25,.75);
	\draw [fill=black](2.75,.75) circle [radius=.05];
	\draw [fill=black](2.25,.75) circle [radius=.05];
	\node at (3.5,.5) {$+$};
	\draw (2+2,0)--(2+3,1);
	\draw (2+2,1)--(2+3,0);
	\draw [red] (2+2.75,.75)--(2+2.25,.75);
	\draw [fill=black](2+2.75,.75) circle [radius=.05];
	\draw [fill=black](2+2.25,.75) circle [radius=.05];
	\draw [red] (4.5,.75) circle [radius=.05];
	\node at (5.5,.5) {$+$};
	\draw (4+2,0)--(4+3,1);
	\draw (4+2,1)--(4+3,0);
	\draw [green] (4+2.75,.75)--(4+2.25,.75);
	\draw [green, ultra thick] (6.5,.65)--(6.5,.85);
	\draw [fill=black](4+2.75,.75) circle [radius=.05];
	\draw [fill=black](4+2.25,.75) circle [radius=.05];
	\end{tikzpicture}.
	\end{equation}
	To show $R\in\mathscr{Q}_{4}$, we only need to show that each term of the right-hand side of Equation \eqref{equ: decomposition of R} belongs to $\mathscr{Q}_4$. Note that the first term in Equation \eqref{equ: decomposition of R} is obtained by the transposition $R$ and $\GHZ$ tensor and thus it belongs to $\mathscr{Q}_4$.  Next, we show that the second term belongs to $\mathscr{Q}_{4}$. Consider the rank-$4$ tensor $B_1$ \raisebox  {-.4cm}{\begin{tikzpicture}
				\draw (0,0)--(1,1);
			\draw (1,0)--(0,1);
			\draw [red, fill=white] (.25,.25) rectangle (.75,.75);
			\foreach \i in {1,...,4}
			{
				\pgfmathsetmacro{\a}{cos(360*(\i-1)/4+45)};
				\pgfmathsetmacro{\b}{sin(360*(\i-1)/4+45)};
				\pgfmathsetmacro{\c}{cos(360*(\i-1)/4+90)};
				\pgfmathsetmacro{\d}{sin(360*(\i-1)/4+90)};
				\draw [fill=black] (.5+.25*1.414*\a,.5+.25*1.414*\b) circle[radius=.05];
			}
			\draw [red] (.25,.5) circle [radius=.05];
			\draw [red] (.75,.5) circle [radius=.05];
			\draw [red] (.5,.25) circle [radius=.05];
			\draw [red] (.5,.75) circle [radius=.05];
		\end{tikzpicture}}$\in\mathscr{Q}_{4}$. Let $\vec{i}=(i_1,i_2,i_3,i_4)$ be an arbitrary tensor in $\mathscr{P}_4$. By definition, we have
	\begin{equation}
	\langle B_1,\vec{i}\rangle= \prod_{i=1}^4 A_\gamma(i_{j},i_{j+1}),
	\end{equation}
	where the indices are modulo over $4$ and the above inner product must be either $1$ or $0$. Suppose $\langle B_A,\vec{i}\rangle=1$. Then $i_j$'s cannot be four distinct numbers otherwise it implies the existence of a square in the Petersen graph. Therefore, we have that 	
	\begin{equation}\label{equ: second term of decomposition of R}
	\begin{tikzpicture}
	\draw (0,0)--(1,1);
	\draw (1,0)--(0,1);
	\draw [red, fill=white] (.25,.25) rectangle (.75,.75);
	\foreach \i in {1,...,4}
	{
		\pgfmathsetmacro{\a}{cos(360*(\i-1)/4+45)};
		\pgfmathsetmacro{\b}{sin(360*(\i-1)/4+45)};
		\pgfmathsetmacro{\c}{cos(360*(\i-1)/4+90)};
		\pgfmathsetmacro{\d}{sin(360*(\i-1)/4+90)};
		\draw [fill=black] (.5+.25*1.414*\a,.5+.25*1.414*\b) circle[radius=.05];
	}
   \draw [red] (.25,.5) circle [radius=.05];
   \draw [red] (.75,.5) circle [radius=.05];
   \draw [red] (.5,.25) circle [radius=.05];
   \draw [red] (.5,.75) circle [radius=.05];
    \node at (1.5,.5) {$=$};
    \draw (2,0)--(3,1);
    \draw [fill=black] (2,1)--(2.25,.75) circle [radius=.05];
    \draw [fill=black] (3,0)--(2.75,.25) circle [radius=.05];
    \node at (3.5,.5) {$+$};
    \draw (4,1)--(5,0);
    \draw [fill=black] (4,0)--(4.25,.25) circle [radius=.05];
    \draw [fill=black] (5,1)--(4.75,.75) circle [radius=.05];
    \node at (5.5,.5) {$+$};
    \draw (6,0)--(7,1);
    \draw (6,1)--(7,0);
    \draw [red] (6.25,.75)--(6.75,.75);
    \draw [red] (6.5,.75) circle [radius=.05];
    \draw [fill=black] (6.25,.75) circle [radius=.05];
    \draw [fill=black] (6.75,.75) circle [radius=.05];
	\end{tikzpicture}.
	\end{equation}		
	Note that the left-hand side of Equation \eqref{equ: second term of decomposition of R} belongs to $\mathscr{Q}_\bullet$ and so do the first two terms of the right hand side since they are generated by $\GHZ$ tensor. Therefore, we have\raisebox{-.4cm}{
	\begin{tikzpicture}
	\draw (6,0)--(7,1);
	\draw (6,1)--(7,0);
	\draw [red] (6.25,.75)--(6.75,.75);
	\draw [fill=black] (6.25,.75) circle [radius=.05];
	\draw [fill=black] (6.75,.75) circle [radius=.05];
	\draw [red] (6.5,.75) circle [radius=.05];
	\node at (7.5,.5) {$\in\mathscr{Q}_4$};
	\end{tikzpicture}}.
Now we proceed to show the third term of the right hand side of Equation \eqref{equ: decomposition of R} belongs to $\mathscr{Q}_{4}$. Consider the following element $B_{2}$ in $\mathscr{Q}_{4}$,
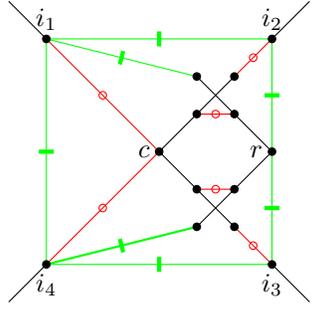
\begin{figure}[H]
	\begin{tikzpicture}
	\draw (0,0)--(.5,.5);
	\draw (0,.5)--(.5,0);
	\draw [red] (0,0)--(.5,0);
	\draw [red] (.25,0) circle [radius=.05];
	\draw [fill=black] (0,0) circle [radius=.05];
	\draw [fill=black] (.5,0) circle [radius=.05];
	\draw (0,-1)--(.5,-1.5);
	\draw (0,-1.5)--(.5,-1);
	\draw [red] (0,-1)--(.5,-1);
	\draw [red] (.25,-1) circle [radius=.05];
	\draw [fill=black] (0,-1) circle [radius=.05];
	\draw [fill=black] (.5,-1) circle [radius=.05];
	\draw [red] (-2,1)--(-.5,-.5)--(-2,-2);
	\draw [red] (-1.25,.25) circle[radius=.05];
	\draw [red] (-1.25,-1.25) circle[radius=.05];
	\draw [fill=black] (0,0)--(-.5,-.5) circle[radius=.05]--(0,-1);	
	\draw [green] (-2,1)--(-2,-2);
	\draw (.5,0)--(1,-.5)--(.5,-1);
	\draw [green] (0,.5)--(-2,1);
	\draw [green, ultra thick] (-1-.05*2.236/5,.75-.05*4*2.236/5)--(-1+.05*2.236/5,.75+.05*4*2.236/5);
	\draw [thick, green] (0,-1.5)--(-2,-2);
	\draw [green, ultra thick] (-1-.05*2.236/5,-1.75+.05*4*2.236/5)--(-1+.05*2.236/5,-1.75-.05*4*2.236/5);
	\draw [red] (.5,.5)--(1,1);
	\draw [red] (.75,.75) circle [radius=.05];
	\draw [red] (.5,-1.5)--(1,-2);
	\draw [red] (.75,-1.75) circle [radius=.05];
	\draw [green] (-2,-2)--(1,-2)--(1,-.5)--(1,1)--(-2,1);
	\draw [green, ultra thick] (-.5,-2.1)--(-.5,-1.9);
	\draw [green, ultra thick] (-.5,1.1)--(-.5,.9);
	\draw [green, ultra thick] (-1.9,-.5)--(-2.1,-.5);
	\draw [green, ultra thick] (.9,-1.25)--(1.1,-1.25);
	\draw [green, ultra thick] (.9,.25)--(1.1,.25);
	\draw [fill=black] (-2,-2) circle[radius=.05];
	\draw [fill=black] (-2,1) circle[radius=.05];
	\draw [fill=black] (1,-2) circle[radius=.05];
	\draw [fill=black] (1,1) circle[radius=.05];
	\draw [fill=black] (.5,.5) circle[radius=.05];
	\draw [fill=black] (0,.5) circle[radius=.05];
	\draw [fill=black] (0,-1.5) circle[radius=.05];
	\draw [fill=black] (.5,-1.5) circle[radius=.05];
	\draw [fill=black] (1,-.5) circle[radius=.05];
	\draw (-2,-2)--(-2.5,-2.5);
	\draw (-2,1)--(-2.5,1.5);
	\draw (1,1)--(1.5,1.5);
	\draw (1,-2)--(1.5,-2.5);
	\node at (-2,-2) [below] {$i_4$};
	\node at (-2,1) [above] {$i_1$};
	\node at (1,-2) [below] {$i_3$};
	\node at (1,1) [above] {$i_2$};
	\node at (1,-.5) [left] {$r$};
	\node at (-.5,-.5) [left] {$c$};
	\end{tikzpicture},
\end{figure}
Let $\vec{i}=(i_1,i_2,i_3,i_4)$ be an arbitrary rank-$4$ tensor in $\mathscr{P}_4$. By definition, we have that $\langle B_2,\vec{i}\rangle$ must take the value in $\{0,1\}$. Suppose $\langle B_2,\vec{i}\rangle=1$. By definition, there exists $c,r\in\{1,2,\cdots,10\}$ such that the following conditions are satisfied,
\begin{enumerate}
	\item\label{itm: con1} $(i_1,i_2)$, $(i_2,i_3)$, $(i_3,i_4)$, $(i_4,i_1)\in E(\Gamma^c)$. 
	\item\label{itm: con2} $(c,r)$, $(c,i_1)$, $(c,i_2)$, $(c,i_3)$, $(c,i_4)\in E(\Gamma)$. 
	\item\label{itm: con3} $(r,i_1)$, $(r,i_2)$, $(r,i_3)$, $(r,i_4)\in E(\Gamma^c)$. 
\end{enumerate}
Now we consider the Kneser construction of the Petersen graph $\Gamma$. Since $B_2\in \mathscr{P}_{4}^{S_5}$, for every $g\in S_5$ we must have $\langle B_2,\vec{i}\rangle=\langle B_2,g\cdot\vec{i}\rangle$. Without of generality, one can assume that $c=\{1,2\}$ and $r=\{3,4\}$. By Condition \eqref{itm: con2} and \eqref{itm: con3}, we have that $i_j\in\{\{3,5\},\{4,5\}\}$ for $1\leq j\leq 4$. Therefore, Condition \eqref{itm: con1} forces that $\vec{i}=(\{3,5\},\{4,5\},\{3,5\},\{4,5\})$ or $\vec{i}=(\{4,5\},\{3,5\},\{4,5\},\{3,5\})$. By definition, we have that $B_2= \raisebox{-.4cm}{\begin{tikzpicture}
		\draw (4+2,0)--(4+3,1);
	\draw (4+2,1)--(4+3,0);
	\draw [green] (4+2.75,.75)--(4+2.25,.75);
	\draw [fill=black](4+2.75,.75) circle [radius=.05];
	\draw [fill=black](4+2.25,.75) circle [radius=.05];
	\draw [green, ultra thick] (6.5,.65)--(6.5,.85);
	\end{tikzpicture}}$ and this implies that \raisebox{-.4cm}{\begin{tikzpicture}
	\draw (4+2,0)--(4+3,1);
	\draw (4+2,1)--(4+3,0);
	\draw [green] (4+2.75,.75)--(4+2.25,.75);
	\draw [fill=black](4+2.75,.75) circle [radius=.05];
	\draw [fill=black](4+2.25,.75) circle [radius=.05];
	\draw [green, ultra thick] (6.5,.65)--(6.5,.85);
	\end{tikzpicture}}$\in\mathscr{Q}_4$.
Therefore, we have that $R\in\mathscr{Q}_4$. By Theorem \ref{thm: 2-box generating}, the group-action model associated with the Petersen graph is generated by the $\GHZ$ tensor and rank-$2$ tensors. 
\end{proof}
\begin{corollary}\label{cor: sbugrouponeway}
	The group subfactor planar algebra for $S_2\times S_3\subset S_5$ is generated by its $2$-boxes.
\end{corollary}
\section{Appendix}\label{sec: Appendix}
Let $\mathscr{P}_\bullet$ be the group-subgroup subfactor planar algebra for $S_2\times S_3\subset S_5$. In this section, we give an explicit construction of a basis for the $4$-box space $\mathscr{P}_{4,+}$. We use \raisebox{-.2cm}{\begin{tikzpicture}
\path [fill=lightgray] (0,0)--(.6,0)--(0,.6)--(.6,.6);
\draw (0,0)--(.6,.6);
\draw (0,.6)--(.6,0);
\end{tikzpicture}} to represent the generator.
\begin{itemize}
	\item The 14 Temperley-Lieb diagrams.
	\item The diagrams containing 1 generator.
	\begin{align*}
		&
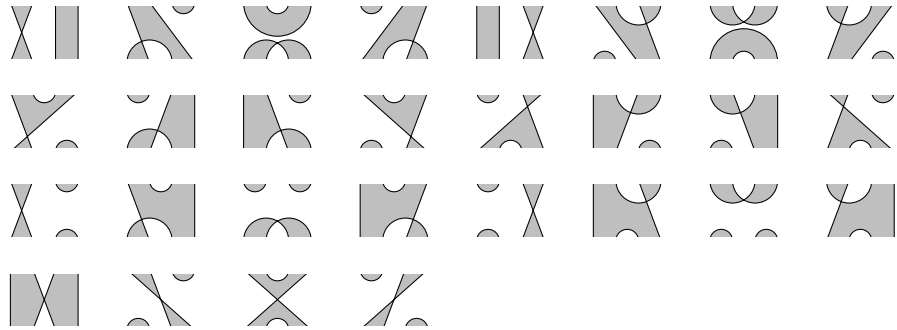
\begin{tikzpicture}
		\draw [fill=lightgray] (0,0.8)--(0.3,0)--(0,0)--(0.3,0.8)--(0,0.8);
		\draw [fill=lightgray] (0.6,0.8)--(0.6,0)--(0.8999999999999999,0)--(0.8999999999999999,0.8)--(0.6,0.8);
		\path [fill=white] (-.05,.05) rectangle (.95,-.05);
		\path [fill=white] (-.05,.85) rectangle (.95,.75);
		\path (1.2,0) rectangle (1.5,.8);
		\end{tikzpicture}
		\begin{tikzpicture}
		\draw [fill=lightgray] (0,0.8)--(0.3,0)--(0,0) arc[radius=0.3, start angle=180, end angle=0]--(0.8999999999999999,0)--(0.3,0.8)--(0,0.8);
		\draw [fill=lightgray] (0.6,0.8) arc[radius=0.15, start angle=180, end angle=360]--(0.6,0.8);
		\path [fill=white] (-.05,.05) rectangle (.95,-.05);
		\path [fill=white] (-.05,.85) rectangle (.95,.75);
		\path (1.2,0) rectangle (1.5,.8);
		\end{tikzpicture}
		\begin{tikzpicture}
		\draw [fill=lightgray] (0,0.8) arc[radius=0.44999999999999996, start angle=180, end angle=360]--(0.6,0.8) arc[radius=0.15, start angle=360, end angle=180]--(0,0.8);
		\draw [fill=lightgray] (0,0) arc[radius=0.3, start angle=180, end angle=0]--(0.8999999999999999,0) arc[radius=0.3, start angle=0, end angle=180]--(0,0);
		\path [fill=white] (-.05,.05) rectangle (.95,-.05);
		\path [fill=white] (-.05,.85) rectangle (.95,.75);
		\path (1.2,0) rectangle (1.5,.8);
		\end{tikzpicture}
		\begin{tikzpicture}
		\draw [fill=lightgray] (0,0.8) arc[radius=0.15, start angle=180, end angle=360]--(0,0.8);
		\draw [fill=lightgray] (0.6,0.8)--(0,0)--(0.3,0) arc[radius=0.3, start angle=180, end angle=0]--(0.6,0)--(0.8999999999999999,0.8)--(0.6,0.8);
		\path [fill=white] (-.05,.05) rectangle (.95,-.05);
		\path [fill=white] (-.05,.85) rectangle (.95,.75);
		\path (1.2,0) rectangle (1.5,.8);
		\end{tikzpicture}
		\begin{tikzpicture}
		\draw [fill=lightgray] (0,0.8)--(0,0)--(0.3,0)--(0.3,0.8)--(0,0.8);
		\draw [fill=lightgray] (0.6,0.8)--(0.8999999999999999,0)--(0.6,0)--(0.8999999999999999,0.8)--(0.6,0.8);
		\path [fill=white] (-.05,.05) rectangle (.95,-.05);
		\path [fill=white] (-.05,.85) rectangle (.95,.75);
		\path (1.2,0) rectangle (1.5,.8);
		\end{tikzpicture}
		\begin{tikzpicture}
		\draw [fill=lightgray] (0,0.8)--(0.6,0)--(0.8999999999999999,0)--(0.6,0.8)--(0.8999999999999999,0.8) arc[radius=0.3, start angle=360, end angle=180]--(0,0.8);
		\draw [fill=lightgray] (0,0) arc[radius=0.15, start angle=180, end angle=0]--(0,0);
		\path [fill=white] (-.05,.05) rectangle (.95,-.05);
		\path [fill=white] (-.05,.85) rectangle (.95,.75);
		\path (1.2,0) rectangle (1.5,.8);
		\end{tikzpicture}
		\begin{tikzpicture}
		\draw [fill=lightgray] (0,0.8) arc[radius=0.3, start angle=180, end angle=360]--(0.8999999999999999,0.8) arc[radius=0.3, start angle=360, end angle=180]--(0,0.8);
		\draw [fill=lightgray] (0,0) arc[radius=0.44999999999999996, start angle=180, end angle=0]--(0.6,0) arc[radius=0.15, start angle=0, end angle=180]--(0,0);
		\path [fill=white] (-.05,.05) rectangle (.95,-.05);
		\path [fill=white] (-.05,.85) rectangle (.95,.75);
		\path (1.2,0) rectangle (1.5,.8);
		\end{tikzpicture}
		\begin{tikzpicture}
		\draw [fill=lightgray] (0,0.8) arc[radius=0.3, start angle=180, end angle=360]--(0.8999999999999999,0.8)--(0.3,0)--(0,0)--(0.3,0.8)--(0,0.8);
		\draw [fill=lightgray] (0.6,0) arc[radius=0.15, start angle=180, end angle=0]--(0.6,0);
		\path [fill=white] (-.05,.05) rectangle (.95,-.05);
		\path [fill=white] (-.05,.85) rectangle (.95,.75);
		\path (1.2,0) rectangle (1.5,.8);
		\end{tikzpicture}\\
		&\begin{tikzpicture}
		\draw [fill=lightgray] (0,0.8)--(0.3,0)--(0,0)--(0.8999999999999999,0.8)--(0.6,0.8) arc[radius=0.15, start angle=360, end angle=180]--(0,0.8);
		\draw [fill=lightgray] (0.6,0) arc[radius=0.15, start angle=180, end angle=0]--(0.6,0);
		\path [fill=white] (-.05,.05) rectangle (.95,-.05);
		\path [fill=white] (-.05,.85) rectangle (.95,.75);
		\path (1.2,0) rectangle (1.5,.8);
		\end{tikzpicture}
		\begin{tikzpicture}
		\draw [fill=lightgray] (0,0.8) arc[radius=0.15, start angle=180, end angle=360]--(0,0.8);
		\draw [fill=lightgray] (0.6,0.8)--(0.3,0)--(0,0) arc[radius=0.3, start angle=180, end angle=0]--(0.8999999999999999,0)--(0.8999999999999999,0.8)--(0.6,0.8);
		\path [fill=white] (-.05,.05) rectangle (.95,-.05);
		\path [fill=white] (-.05,.85) rectangle (.95,.75);
		\path (1.2,0) rectangle (1.5,.8);
		\end{tikzpicture}
		\begin{tikzpicture}
		\draw [fill=lightgray] (0,0.8)--(0,0)--(0.3,0) arc[radius=0.3, start angle=180, end angle=0]--(0.6,0)--(0.3,0.8)--(0,0.8);
		\draw [fill=lightgray] (0.6,0.8) arc[radius=0.15, start angle=180, end angle=360]--(0.6,0.8);
		\path [fill=white] (-.05,.05) rectangle (.95,-.05);
		\path [fill=white] (-.05,.85) rectangle (.95,.75);
		\path (1.2,0) rectangle (1.5,.8);
		\end{tikzpicture}
		\begin{tikzpicture}
		\draw [fill=lightgray] (0,0.8)--(0.8999999999999999,0)--(0.6,0)--(0.8999999999999999,0.8)--(0.6,0.8) arc[radius=0.15, start angle=360, end angle=180]--(0,0.8);
		\draw [fill=lightgray] (0,0) arc[radius=0.15, start angle=180, end angle=0]--(0,0);
		\path [fill=white] (-.05,.05) rectangle (.95,-.05);
		\path [fill=white] (-.05,.85) rectangle (.95,.75);
		\path (1.2,0) rectangle (1.5,.8);
		\end{tikzpicture}
		\begin{tikzpicture}
		\draw [fill=lightgray] (0,0.8) arc[radius=0.15, start angle=180, end angle=360]--(0,0.8);
		\draw [fill=lightgray] (0.6,0.8)--(0.8999999999999999,0)--(0.6,0) arc[radius=0.15, start angle=0, end angle=180]--(0,0)--(0.8999999999999999,0.8)--(0.6,0.8);
		\path [fill=white] (-.05,.05) rectangle (.95,-.05);
		\path [fill=white] (-.05,.85) rectangle (.95,.75);
		\path (1.2,0) rectangle (1.5,.8);
		\end{tikzpicture}
		\begin{tikzpicture}
		\draw [fill=lightgray] (0,0.8)--(0,0)--(0.3,0)--(0.6,0.8)--(0.8999999999999999,0.8) arc[radius=0.3, start angle=360, end angle=180]--(0,0.8);
		\draw [fill=lightgray] (0.6,0) arc[radius=0.15, start angle=180, end angle=0]--(0.6,0);
		\path [fill=white] (-.05,.05) rectangle (.95,-.05);
		\path [fill=white] (-.05,.85) rectangle (.95,.75);
		\path (1.2,0) rectangle (1.5,.8);
		\end{tikzpicture}
		\begin{tikzpicture}
		\draw [fill=lightgray] (0,0.8) arc[radius=0.3, start angle=180, end angle=360]--(0.8999999999999999,0.8)--(0.8999999999999999,0)--(0.6,0)--(0.3,0.8)--(0,0.8);
		\draw [fill=lightgray] (0,0) arc[radius=0.15, start angle=180, end angle=0]--(0,0);
		\path [fill=white] (-.05,.05) rectangle (.95,-.05);
		\path [fill=white] (-.05,.85) rectangle (.95,.75);
		\path (1.2,0) rectangle (1.5,.8);
		\end{tikzpicture}
		\begin{tikzpicture}
		\draw [fill=lightgray] (0,0.8)--(0.8999999999999999,0)--(0.6,0) arc[radius=0.15, start angle=0, end angle=180]--(0,0)--(0.3,0.8)--(0,0.8);
		\draw [fill=lightgray] (0.6,0.8) arc[radius=0.15, start angle=180, end angle=360]--(0.6,0.8);
		\path [fill=white] (-.05,.05) rectangle (.95,-.05);
		\path [fill=white] (-.05,.85) rectangle (.95,.75);
		\path (1.2,0) rectangle (1.5,.8);
		\end{tikzpicture}\\
		&\begin{tikzpicture}
		\draw [fill=lightgray] (0,0.8)--(0.3,0)--(0,0)--(0.3,0.8)--(0,0.8);
		\draw [fill=lightgray] (0.6,0.8) arc[radius=0.15, start angle=180, end angle=360]--(0.6,0.8);
		\draw [fill=lightgray] (0.6,0) arc[radius=0.15, start angle=180, end angle=0]--(0.6,0);
		\path [fill=white] (-.05,.05) rectangle (.95,-.05);
		\path [fill=white] (-.05,.85) rectangle (.95,.75);
		\path (1.2,0) rectangle (1.5,.8);
		\end{tikzpicture}
		\begin{tikzpicture}
		\draw [fill=lightgray] (0,0.8)--(0.3,0)--(0,0) arc[radius=0.3, start angle=180, end angle=0]--(0.8999999999999999,0)--(0.8999999999999999,0.8)--(0.6,0.8) arc[radius=0.15, start angle=360, end angle=180]--(0,0.8);
		\path [fill=white] (-.05,.05) rectangle (.95,-.05);
		\path [fill=white] (-.05,.85) rectangle (.95,.75);
		\path (1.2,0) rectangle (1.5,.8);
		\end{tikzpicture}
		\begin{tikzpicture}
		\draw [fill=lightgray] (0,0.8) arc[radius=0.15, start angle=180, end angle=360]--(0,0.8);
		\draw [fill=lightgray] (0.6,0.8) arc[radius=0.15, start angle=180, end angle=360]--(0.6,0.8);
		\draw [fill=lightgray] (0,0) arc[radius=0.3, start angle=180, end angle=0]--(0.8999999999999999,0) arc[radius=0.3, start angle=0, end angle=180]--(0,0);
		\path [fill=white] (-.05,.05) rectangle (.95,-.05);
		\path [fill=white] (-.05,.85) rectangle (.95,.75);
		\path (1.2,0) rectangle (1.5,.8);
		\end{tikzpicture}
		\begin{tikzpicture}
		\draw [fill=lightgray] (0,0.8)--(0,0)--(0.3,0) arc[radius=0.3, start angle=180, end angle=0]--(0.6,0)--(0.8999999999999999,0.8)--(0.6,0.8) arc[radius=0.15, start angle=360, end angle=180]--(0,0.8);
		\path [fill=white] (-.05,.05) rectangle (.95,-.05);
		\path [fill=white] (-.05,.85) rectangle (.95,.75);
		\path (1.2,0) rectangle (1.5,.8);
		\end{tikzpicture}
		\begin{tikzpicture}
		\draw [fill=lightgray] (0,0.8) arc[radius=0.15, start angle=180, end angle=360]--(0,0.8);
		\draw [fill=lightgray] (0.6,0.8)--(0.8999999999999999,0)--(0.6,0)--(0.8999999999999999,0.8)--(0.6,0.8);
		\draw [fill=lightgray] (0,0) arc[radius=0.15, start angle=180, end angle=0]--(0,0);
		\path [fill=white] (-.05,.05) rectangle (.95,-.05);
		\path [fill=white] (-.05,.85) rectangle (.95,.75);
		\path (1.2,0) rectangle (1.5,.8);
		\end{tikzpicture}
		\begin{tikzpicture}
		\draw [fill=lightgray] (0,0.8)--(0,0)--(0.3,0) arc[radius=0.15, start angle=180, end angle=0]--(0.8999999999999999,0)--(0.6,0.8)--(0.8999999999999999,0.8) arc[radius=0.3, start angle=360, end angle=180]--(0,0.8);
		\path [fill=white] (-.05,.05) rectangle (.95,-.05);
		\path [fill=white] (-.05,.85) rectangle (.95,.75);
		\path (1.2,0) rectangle (1.5,.8);
		\end{tikzpicture}
		\begin{tikzpicture}
		\draw [fill=lightgray] (0,0.8) arc[radius=0.3, start angle=180, end angle=360]--(0.8999999999999999,0.8) arc[radius=0.3, start angle=360, end angle=180]--(0,0.8);
		\draw [fill=lightgray] (0,0) arc[radius=0.15, start angle=180, end angle=0]--(0,0);
		\draw [fill=lightgray] (0.6,0) arc[radius=0.15, start angle=180, end angle=0]--(0.6,0);
		\path [fill=white] (-.05,.05) rectangle (.95,-.05);
		\path [fill=white] (-.05,.85) rectangle (.95,.75);
		\path (1.2,0) rectangle (1.5,.8);
		\end{tikzpicture}
		\begin{tikzpicture}
		\draw [fill=lightgray] (0,0.8) arc[radius=0.3, start angle=180, end angle=360]--(0.8999999999999999,0.8)--(0.8999999999999999,0)--(0.6,0) arc[radius=0.15, start angle=0, end angle=180]--(0,0)--(0.3,0.8)--(0,0.8);
		\path [fill=white] (-.05,.05) rectangle (.95,-.05);
		\path [fill=white] (-.05,.85) rectangle (.95,.75);
		\path (1.2,0) rectangle (1.5,.8);
		\end{tikzpicture}\\
		&\begin{tikzpicture}
		\draw [fill=lightgray] (0,0.8)--(0,0)--(0.3,0)--(0.6,0.8)--(0.8999999999999999,0.8)--(0.8999999999999999,0)--(0.6,0)--(0.3,0.8)--(0,0.8);
		\path [fill=white] (-.05,.05) rectangle (.95,-.05);
		\path [fill=white] (-.05,.85) rectangle (.95,.75);
		\path (1.2,0) rectangle (1.5,.8);
		\end{tikzpicture}
		\begin{tikzpicture}
		\draw [fill=lightgray] (0,0.8)--(0.8999999999999999,0)--(0.6,0)--(0.3,0.8)--(0,0.8);
		\draw [fill=lightgray] (0.6,0.8) arc[radius=0.15, start angle=180, end angle=360]--(0.6,0.8);
		\draw [fill=lightgray] (0,0) arc[radius=0.15, start angle=180, end angle=0]--(0,0);
		\path [fill=white] (-.05,.05) rectangle (.95,-.05);
		\path [fill=white] (-.05,.85) rectangle (.95,.75);
		\path (1.2,0) rectangle (1.5,.8);
		\end{tikzpicture}
		\begin{tikzpicture}
		\draw [fill=lightgray] (0,0.8)--(0.8999999999999999,0)--(0.6,0) arc[radius=0.15, start angle=0, end angle=180]--(0,0)--(0.8999999999999999,0.8)--(0.6,0.8) arc[radius=0.15, start angle=360, end angle=180]--(0,0.8);
		\path [fill=white] (-.05,.05) rectangle (.95,-.05);
		\path [fill=white] (-.05,.85) rectangle (.95,.75);
		\path (1.2,0) rectangle (1.5,.8);
		\end{tikzpicture}
		\begin{tikzpicture}
		\draw [fill=lightgray] (0,0.8) arc[radius=0.15, start angle=180, end angle=360]--(0,0.8);
		\draw [fill=lightgray] (0.6,0.8)--(0.3,0)--(0,0)--(0.8999999999999999,0.8)--(0.6,0.8);
		\draw [fill=lightgray] (0.6,0) arc[radius=0.15, start angle=180, end angle=0]--(0.6,0);
		\path [fill=white] (-.05,.05) rectangle (.95,-.05);
		\path [fill=white] (-.05,.85) rectangle (.95,.75);
		\path (1.2,0) rectangle (1.5,.8);
		\end{tikzpicture}
	\end{align*}	
	\item The diagrams containing 2 generator.
	\begin{align*}
	&
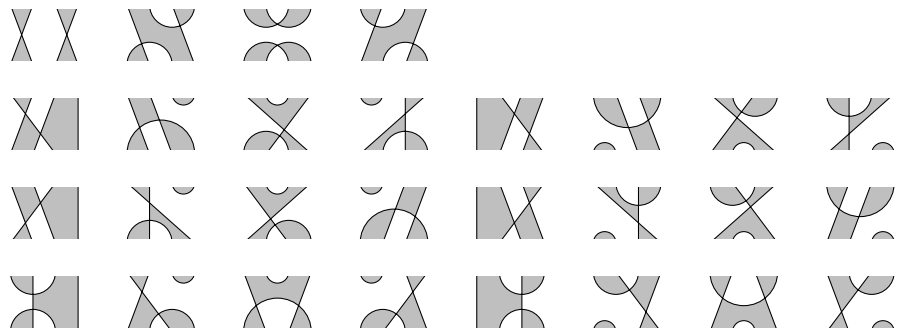
\begin{tikzpicture}
	\draw [fill=lightgray] (0,0.8)--(0.3,0)--(0,0)--(0.3,0.8)--(0,0.8);
	\draw [fill=lightgray] (0.6,0.8)--(0.8999999999999999,0)--(0.6,0)--(0.8999999999999999,0.8)--(0.6,0.8);
	\path [fill=white] (-.05,.05) rectangle (.95,-.05);
	\path [fill=white] (-.05,.85) rectangle (.95,.75);
	\path (1.2,0) rectangle (1.5,.8);
	\end{tikzpicture}
	\begin{tikzpicture}
	\draw [fill=lightgray] (0,0.8)--(0.3,0)--(0,0) arc[radius=0.3, start angle=180, end angle=0]--(0.8999999999999999,0)--(0.6,0.8)--(0.8999999999999999,0.8) arc[radius=0.3, start angle=360, end angle=180]--(0,0.8);
	\path [fill=white] (-.05,.05) rectangle (.95,-.05);
	\path [fill=white] (-.05,.85) rectangle (.95,.75);
	\path (1.2,0) rectangle (1.5,.8);
	\end{tikzpicture}
	\begin{tikzpicture}
	\draw [fill=lightgray] (0,0.8) arc[radius=0.3, start angle=180, end angle=360]--(0.8999999999999999,0.8) arc[radius=0.3, start angle=360, end angle=180]--(0,0.8);
	\draw [fill=lightgray] (0,0) arc[radius=0.3, start angle=180, end angle=0]--(0.8999999999999999,0) arc[radius=0.3, start angle=0, end angle=180]--(0,0);
	\path [fill=white] (-.05,.05) rectangle (.95,-.05);
	\path [fill=white] (-.05,.85) rectangle (.95,.75);
	\path (1.2,0) rectangle (1.5,.8);
	\end{tikzpicture}
	\begin{tikzpicture}
	\draw [fill=lightgray] (0,0.8) arc[radius=0.3, start angle=180, end angle=360]--(0.8999999999999999,0.8)--(0.6,0)--(0.8999999999999999,0) arc[radius=0.3, start angle=0, end angle=180]--(0,0)--(0.3,0.8)--(0,0.8);
	\path [fill=white] (-.05,.05) rectangle (.95,-.05);
	\path [fill=white] (-.05,.85) rectangle (.95,.75);
	\path (1.2,0) rectangle (1.5,.8);
	\end{tikzpicture}\\
	&\begin{tikzpicture}
	\draw [fill=lightgray] (0,0.8)--(0.6,0)--(0.8999999999999999,0)--(0.8999999999999999,0.8)--(0.6,0.8)--(0.3,0)--(0,0)--(0.3,0.8)--(0,0.8);
	\path [fill=white] (-.05,.05) rectangle (.95,-.05);
	\path [fill=white] (-.05,.85) rectangle (.95,.75);
	\path (1.2,0) rectangle (1.5,.8);
	\end{tikzpicture}
	\begin{tikzpicture}
	\draw [fill=lightgray] (0,0.8)--(0.3,0)--(0,0) arc[radius=0.44999999999999996, start angle=180, end angle=0]--(0.6,0)--(0.3,0.8)--(0,0.8);
	\draw [fill=lightgray] (0.6,0.8) arc[radius=0.15, start angle=180, end angle=360]--(0.6,0.8);
	\path [fill=white] (-.05,.05) rectangle (.95,-.05);
	\path [fill=white] (-.05,.85) rectangle (.95,.75);
	\path (1.2,0) rectangle (1.5,.8);
	\end{tikzpicture}
	\begin{tikzpicture}
	\draw [fill=lightgray] (0,0.8)--(0.8999999999999999,0)--(0.6,0) arc[radius=0.3, start angle=0, end angle=180]--(0.3,0)--(0.8999999999999999,0.8)--(0.6,0.8) arc[radius=0.15, start angle=360, end angle=180]--(0,0.8);
	\path [fill=white] (-.05,.05) rectangle (.95,-.05);
	\path [fill=white] (-.05,.85) rectangle (.95,.75);
	\path (1.2,0) rectangle (1.5,.8);
	\end{tikzpicture}
	\begin{tikzpicture}
	\draw [fill=lightgray] (0,0.8) arc[radius=0.15, start angle=180, end angle=360]--(0,0.8);
	\draw [fill=lightgray] (0.6,0.8)--(0.6,0)--(0.8999999999999999,0) arc[radius=0.3, start angle=0, end angle=180]--(0,0)--(0.8999999999999999,0.8)--(0.6,0.8);
	\path [fill=white] (-.05,.05) rectangle (.95,-.05);
	\path [fill=white] (-.05,.85) rectangle (.95,.75);
	\path (1.2,0) rectangle (1.5,.8);
	\end{tikzpicture}
	\begin{tikzpicture}
	\draw [fill=lightgray] (0,0.8)--(0,0)--(0.3,0)--(0.6,0.8)--(0.8999999999999999,0.8)--(0.6,0)--(0.8999999999999999,0)--(0.3,0.8)--(0,0.8);
	\path [fill=white] (-.05,.05) rectangle (.95,-.05);
	\path [fill=white] (-.05,.85) rectangle (.95,.75);
	\path (1.2,0) rectangle (1.5,.8);
	\end{tikzpicture}
	\begin{tikzpicture}
	\draw [fill=lightgray] (0,0.8) arc[radius=0.44999999999999996, start angle=180, end angle=360]--(0.6,0.8)--(0.8999999999999999,0)--(0.6,0)--(0.3,0.8)--(0,0.8);
	\draw [fill=lightgray] (0,0) arc[radius=0.15, start angle=180, end angle=0]--(0,0);
	\path [fill=white] (-.05,.05) rectangle (.95,-.05);
	\path [fill=white] (-.05,.85) rectangle (.95,.75);
	\path (1.2,0) rectangle (1.5,.8);
	\end{tikzpicture}
	\begin{tikzpicture}
	\draw [fill=lightgray] (0,0.8)--(0.8999999999999999,0)--(0.6,0) arc[radius=0.15, start angle=0, end angle=180]--(0,0)--(0.6,0.8)--(0.8999999999999999,0.8) arc[radius=0.3, start angle=360, end angle=180]--(0,0.8);
	\path [fill=white] (-.05,.05) rectangle (.95,-.05);
	\path [fill=white] (-.05,.85) rectangle (.95,.75);
	\path (1.2,0) rectangle (1.5,.8);
	\end{tikzpicture}
	\begin{tikzpicture}
	\draw [fill=lightgray] (0,0.8) arc[radius=0.3, start angle=180, end angle=360]--(0.8999999999999999,0.8)--(0,0)--(0.3,0)--(0.3,0.8)--(0,0.8);
	\draw [fill=lightgray] (0.6,0) arc[radius=0.15, start angle=180, end angle=0]--(0.6,0);
	\path [fill=white] (-.05,.05) rectangle (.95,-.05);
	\path [fill=white] (-.05,.85) rectangle (.95,.75);
	\path (1.2,0) rectangle (1.5,.8);
	\end{tikzpicture}\\
	&\begin{tikzpicture}
	\draw [fill=lightgray] (0,0.8)--(0.3,0)--(0,0)--(0.6,0.8)--(0.8999999999999999,0.8)--(0.8999999999999999,0)--(0.6,0)--(0.3,0.8)--(0,0.8);
	\path [fill=white] (-.05,.05) rectangle (.95,-.05);
	\path [fill=white] (-.05,.85) rectangle (.95,.75);
	\path (1.2,0) rectangle (1.5,.8);
	\end{tikzpicture}
	\begin{tikzpicture}
	\draw [fill=lightgray] (0,0.8)--(0.8999999999999999,0)--(0.6,0) arc[radius=0.3, start angle=0, end angle=180]--(0.3,0)--(0.3,0.8)--(0,0.8);
	\draw [fill=lightgray] (0.6,0.8) arc[radius=0.15, start angle=180, end angle=360]--(0.6,0.8);
	\path [fill=white] (-.05,.05) rectangle (.95,-.05);
	\path [fill=white] (-.05,.85) rectangle (.95,.75);
	\path (1.2,0) rectangle (1.5,.8);
	\end{tikzpicture}
	\begin{tikzpicture}
	\draw [fill=lightgray] (0,0.8)--(0.6,0)--(0.8999999999999999,0) arc[radius=0.3, start angle=0, end angle=180]--(0,0)--(0.8999999999999999,0.8)--(0.6,0.8) arc[radius=0.15, start angle=360, end angle=180]--(0,0.8);
	\path [fill=white] (-.05,.05) rectangle (.95,-.05);
	\path [fill=white] (-.05,.85) rectangle (.95,.75);
	\path (1.2,0) rectangle (1.5,.8);
	\end{tikzpicture}
	\begin{tikzpicture}
	\draw [fill=lightgray] (0,0.8) arc[radius=0.15, start angle=180, end angle=360]--(0,0.8);
	\draw [fill=lightgray] (0.6,0.8)--(0.3,0)--(0,0) arc[radius=0.44999999999999996, start angle=180, end angle=0]--(0.6,0)--(0.8999999999999999,0.8)--(0.6,0.8);
	\path [fill=white] (-.05,.05) rectangle (.95,-.05);
	\path [fill=white] (-.05,.85) rectangle (.95,.75);
	\path (1.2,0) rectangle (1.5,.8);
	\end{tikzpicture}
	\begin{tikzpicture}
	\draw [fill=lightgray] (0,0.8)--(0,0)--(0.3,0)--(0.8999999999999999,0.8)--(0.6,0.8)--(0.8999999999999999,0)--(0.6,0)--(0.3,0.8)--(0,0.8);
	\path [fill=white] (-.05,.05) rectangle (.95,-.05);
	\path [fill=white] (-.05,.85) rectangle (.95,.75);
	\path (1.2,0) rectangle (1.5,.8);
	\end{tikzpicture}
	\begin{tikzpicture}
	\draw [fill=lightgray] (0,0.8)--(0.8999999999999999,0)--(0.6,0)--(0.6,0.8)--(0.8999999999999999,0.8) arc[radius=0.3, start angle=360, end angle=180]--(0,0.8);
	\draw [fill=lightgray] (0,0) arc[radius=0.15, start angle=180, end angle=0]--(0,0);
	\path [fill=white] (-.05,.05) rectangle (.95,-.05);
	\path [fill=white] (-.05,.85) rectangle (.95,.75);
	\path (1.2,0) rectangle (1.5,.8);
	\end{tikzpicture}
	\begin{tikzpicture}
	\draw [fill=lightgray] (0,0.8) arc[radius=0.3, start angle=180, end angle=360]--(0.8999999999999999,0.8)--(0,0)--(0.3,0) arc[radius=0.15, start angle=180, end angle=0]--(0.8999999999999999,0)--(0.3,0.8)--(0,0.8);
	\path [fill=white] (-.05,.05) rectangle (.95,-.05);
	\path [fill=white] (-.05,.85) rectangle (.95,.75);
	\path (1.2,0) rectangle (1.5,.8);
	\end{tikzpicture}
	\begin{tikzpicture}
	\draw [fill=lightgray] (0,0.8) arc[radius=0.44999999999999996, start angle=180, end angle=360]--(0.6,0.8)--(0.3,0)--(0,0)--(0.3,0.8)--(0,0.8);
	\draw [fill=lightgray] (0.6,0) arc[radius=0.15, start angle=180, end angle=0]--(0.6,0);
	\path [fill=white] (-.05,.05) rectangle (.95,-.05);
	\path [fill=white] (-.05,.85) rectangle (.95,.75);
	\path (1.2,0) rectangle (1.5,.8);
	\end{tikzpicture}\\
	&\begin{tikzpicture}
	\draw [fill=lightgray] (0,0.8) arc[radius=0.3, start angle=180, end angle=360]--(0.8999999999999999,0.8)--(0.8999999999999999,0)--(0.6,0) arc[radius=0.3, start angle=0, end angle=180]--(0.3,0)--(0.3,0.8)--(0,0.8);
	\path [fill=white] (-.05,.05) rectangle (.95,-.05);
	\path [fill=white] (-.05,.85) rectangle (.95,.75);
	\path (1.2,0) rectangle (1.5,.8);
	\end{tikzpicture}
	\begin{tikzpicture}
	\draw [fill=lightgray] (0,0.8)--(0.6,0)--(0.8999999999999999,0) arc[radius=0.3, start angle=0, end angle=180]--(0,0)--(0.3,0.8)--(0,0.8);
	\draw [fill=lightgray] (0.6,0.8) arc[radius=0.15, start angle=180, end angle=360]--(0.6,0.8);
	\path [fill=white] (-.05,.05) rectangle (.95,-.05);
	\path [fill=white] (-.05,.85) rectangle (.95,.75);
	\path (1.2,0) rectangle (1.5,.8);
	\end{tikzpicture}
	\begin{tikzpicture}
	\draw [fill=lightgray] (0,0.8)--(0.3,0)--(0,0) arc[radius=0.44999999999999996, start angle=180, end angle=0]--(0.6,0)--(0.8999999999999999,0.8)--(0.6,0.8) arc[radius=0.15, start angle=360, end angle=180]--(0,0.8);
	\path [fill=white] (-.05,.05) rectangle (.95,-.05);
	\path [fill=white] (-.05,.85) rectangle (.95,.75);
	\path (1.2,0) rectangle (1.5,.8);
	\end{tikzpicture}
	\begin{tikzpicture}
	\draw [fill=lightgray] (0,0.8) arc[radius=0.15, start angle=180, end angle=360]--(0,0.8);
	\draw [fill=lightgray] (0.6,0.8)--(0.8999999999999999,0)--(0.6,0) arc[radius=0.3, start angle=0, end angle=180]--(0.3,0)--(0.8999999999999999,0.8)--(0.6,0.8);
	\path [fill=white] (-.05,.05) rectangle (.95,-.05);
	\path [fill=white] (-.05,.85) rectangle (.95,.75);
	\path (1.2,0) rectangle (1.5,.8);
	\end{tikzpicture}
	\begin{tikzpicture}
	\draw [fill=lightgray] (0,0.8)--(0,0)--(0.3,0) arc[radius=0.3, start angle=180, end angle=0]--(0.6,0)--(0.6,0.8)--(0.8999999999999999,0.8) arc[radius=0.3, start angle=360, end angle=180]--(0,0.8);
	\path [fill=white] (-.05,.05) rectangle (.95,-.05);
	\path [fill=white] (-.05,.85) rectangle (.95,.75);
	\path (1.2,0) rectangle (1.5,.8);
	\end{tikzpicture}
	\begin{tikzpicture}
	\draw [fill=lightgray] (0,0.8) arc[radius=0.3, start angle=180, end angle=360]--(0.8999999999999999,0.8)--(0.6,0)--(0.8999999999999999,0)--(0.3,0.8)--(0,0.8);
	\draw [fill=lightgray] (0,0) arc[radius=0.15, start angle=180, end angle=0]--(0,0);
	\path [fill=white] (-.05,.05) rectangle (.95,-.05);
	\path [fill=white] (-.05,.85) rectangle (.95,.75);
	\path (1.2,0) rectangle (1.5,.8);
	\end{tikzpicture}
	\begin{tikzpicture}
	\draw [fill=lightgray] (0,0.8) arc[radius=0.44999999999999996, start angle=180, end angle=360]--(0.6,0.8)--(0.8999999999999999,0)--(0.6,0) arc[radius=0.15, start angle=0, end angle=180]--(0,0)--(0.3,0.8)--(0,0.8);
	\path [fill=white] (-.05,.05) rectangle (.95,-.05);
	\path [fill=white] (-.05,.85) rectangle (.95,.75);
	\path (1.2,0) rectangle (1.5,.8);
	\end{tikzpicture}
	\begin{tikzpicture}
	\draw [fill=lightgray] (0,0.8)--(0.3,0)--(0,0)--(0.6,0.8)--(0.8999999999999999,0.8) arc[radius=0.3, start angle=360, end angle=180]--(0,0.8);
	\draw [fill=lightgray] (0.6,0) arc[radius=0.15, start angle=180, end angle=0]--(0.6,0);
	\path [fill=white] (-.05,.05) rectangle (.95,-.05);
	\path [fill=white] (-.05,.85) rectangle (.95,.75);
	\path (1.2,0) rectangle (1.5,.8);
	\end{tikzpicture}
	\end{align*}
	\item The diagrams containing 3 generator.
	\begin{align*}
	&
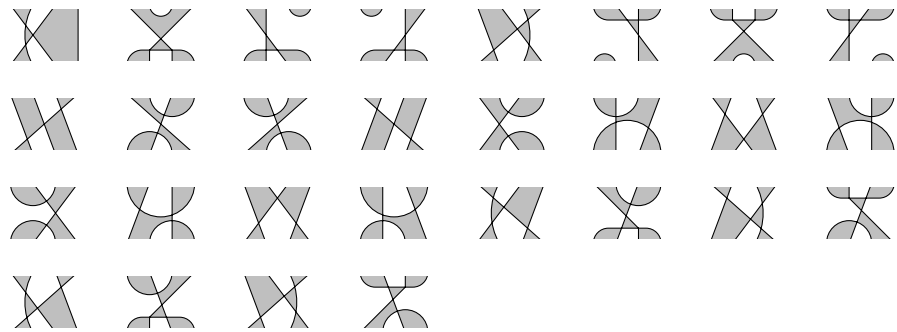
\begin{tikzpicture}
	\draw [fill=lightgray] (0,0)--(.6,.8)--(.9,.8)--(.9,0)--(.6,0)--(0,.8)--(.3,.8) arc [radius=.8, start angle=150, end angle=210]--(0,0);
	\path [fill=white] (-.05,.05) rectangle (.95,-.05);
	\path [fill=white] (-.05,.85) rectangle (.95,.75);
	\path (1.2,0) rectangle (1.5,.8);
	\end{tikzpicture}
	\begin{tikzpicture}
	\draw [fill=lightgray] (0,0) arc [radius=.2, start angle=180, end angle=90]--(.7,.2) arc[radius=.2, start angle=90, end angle=0]--(.6,0)--(.6,.2)--(0,.8)--(.3,.8) arc[radius=.15, start angle=180, end angle=360]--(.9,.8)--(.3,.2)--(.3,.0);
	\path [fill=white] (-.05,.05) rectangle (.95,-.05);
	\path [fill=white] (-.05,.85) rectangle (.95,.75);
	\path (1.2,0) rectangle (1.5,.8);
	\end{tikzpicture}
	\begin{tikzpicture}
	\draw [fill=lightgray] (0,0.8)--(0.6,0)--(0.8999999999999999,0) arc[radius=0.2, start angle=0, end angle=90]--(.7,.2)--(.2,.2) arc[radius=.2, start angle=90, end angle=180]--(0.3,0)--(.3,.8)--(0,0.8);
	\draw [fill=lightgray] (0.6,0.8) arc[radius=0.15, start angle=180, end angle=360]--(0.6,0.8);
	\path [fill=white] (-.05,.05) rectangle (.95,-.05);
	\path [fill=white] (-.05,.85) rectangle (.95,.75);
	\path (1.2,0) rectangle (1.5,.8);
	\end{tikzpicture}
	\begin{tikzpicture}
	\draw [fill=lightgray] (0,0.8) arc[radius=0.15, start angle=180, end angle=360]--(0,0.8);
	\draw [fill=lightgray] (0.6,0.8)--(0.6,0)--(0.8999999999999999,0) arc[radius=0.2, start angle=0, end angle=90]--(.2,.2) arc[radius=.2, start angle=90, end angle=180]--(0.3,0)--(0.8999999999999999,0.8)--(0.6,0.8);
	\path [fill=white] (-.05,.05) rectangle (.95,-.05);
	\path [fill=white] (-.05,.85) rectangle (.95,.75);
	\path (1.2,0) rectangle (1.5,.8);
	\end{tikzpicture}
	\begin{tikzpicture}
	\draw [fill=lightgray] (0,0.8)--(0.3,0)--(0,0)--(0.8999999999999999,0.8)--(0.6,0.8) arc[radius=.8, start angle=30, end angle=-30]--(0.8999999999999999,0)--(0.3,0.8)--(0,0.8);
	\path [fill=white] (-.05,.05) rectangle (.95,-.05);
	\path [fill=white] (-.05,.85) rectangle (.95,.75);
	\path (1.2,0) rectangle (1.5,.8);
	\end{tikzpicture}
	\begin{tikzpicture}
	\draw [fill=lightgray] (0,0.8) arc[radius=0.2, start angle=180, end angle=270]--(.7,.6) arc[radius=.2, start angle=270, end angle=360]--(0.6,0.8)--(0.6,0)--(0.8999999999999999,0)--(0.3,0.8)--(0,0.8);
	\draw [fill=lightgray] (0,0) arc[radius=0.15, start angle=180, end angle=0]--(0,0);
	\path [fill=white] (-.05,.05) rectangle (.95,-.05);
	\path [fill=white] (-.05,.85) rectangle (.95,.75);
	\path (1.2,0) rectangle (1.5,.8);
	\end{tikzpicture}
	\begin{tikzpicture}
	\draw [fill=lightgray] (0,0.8) arc[radius=0.2, start angle=180, end angle=270]--(.7,.6) arc[radius=.2, start angle=270, end angle=360]--(0.6,0.8)--(.6,.6)--(0,0)--(0.3,0) arc[radius=0.15, start angle=180, end angle=0]--(0.8999999999999999,0)--(0.3,0.6)--(.3,.8)--(0,0.8);
	\path [fill=white] (-.05,.05) rectangle (.95,-.05);
	\path [fill=white] (-.05,.85) rectangle (.95,.75);
	\path (1.2,0) rectangle (1.5,.8);
	\end{tikzpicture}
	\begin{tikzpicture}
	\draw [fill=lightgray] (0,0.8) arc[radius=0.2, start angle=180, end angle=270]--(.7,.6) arc[radius=.2, start angle=270, end angle=360]--(0.6,0.8)--(0,0)--(0.3,0)--(0.3,0.8)--(0,0.8);
	\draw [fill=lightgray] (0.6,0) arc[radius=0.15, start angle=180, end angle=0]--(0.6,0);
	\path [fill=white] (-.05,.05) rectangle (.95,-.05);
	\path [fill=white] (-.05,.85) rectangle (.95,.75);
	\path (1.2,0) rectangle (1.5,.8);
	\end{tikzpicture}\\
	&\begin{tikzpicture}
	\draw [fill=lightgray] (0,0.8)--(0.3,0)--(0,0)--(0.8999999999999999,0.8)--(0.6,0.8)--(0.8999999999999999,0)--(0.6,0)--(0.3,0.8)--(0,0.8);
	\path [fill=white] (-.05,.05) rectangle (.95,-.05);
	\path [fill=white] (-.05,.85) rectangle (.95,.75);
	\path (1.2,0) rectangle (1.5,.8);
	\end{tikzpicture}
	\begin{tikzpicture}
	\draw [fill=lightgray] (0,0.8)--(0.8999999999999999,0)--(0.6,0) arc[radius=0.3, start angle=0, end angle=180]--(0.3,0)--(0.6,0.8)--(0.8999999999999999,0.8) arc[radius=0.3, start angle=360, end angle=180]--(0,0.8);
	\path [fill=white] (-.05,.05) rectangle (.95,-.05);
	\path [fill=white] (-.05,.85) rectangle (.95,.75);
	\path (1.2,0) rectangle (1.5,.8);
	\end{tikzpicture}
	\begin{tikzpicture}
	\draw [fill=lightgray] (0,0.8) arc[radius=0.3, start angle=180, end angle=360]--(0.8999999999999999,0.8)--(0,0)--(0.3,0) arc[radius=0.3, start angle=180, end angle=0]--(0.6,0)--(0.3,0.8)--(0,0.8);
	\path [fill=white] (-.05,.05) rectangle (.95,-.05);
	\path [fill=white] (-.05,.85) rectangle (.95,.75);
	\path (1.2,0) rectangle (1.5,.8);
	\end{tikzpicture}
	\begin{tikzpicture}
	\draw [fill=lightgray] (0,0.8)--(0.8999999999999999,0)--(0.6,0)--(0.8999999999999999,0.8)--(0.6,0.8)--(0.3,0)--(0,0)--(0.3,0.8)--(0,0.8);
	\path [fill=white] (-.05,.05) rectangle (.95,-.05);
	\path [fill=white] (-.05,.85) rectangle (.95,.75);
	\path (1.2,0) rectangle (1.5,.8);
	\end{tikzpicture}
	\begin{tikzpicture}
	\draw [fill=lightgray] (0,0.8)--(0.6,0)--(0.8999999999999999,0) arc[radius=0.3, start angle=0, end angle=180]--(0,0)--(0.6,0.8)--(0.8999999999999999,0.8) arc[radius=0.3, start angle=360, end angle=180]--(0,0.8);
	\path [fill=white] (-.05,.05) rectangle (.95,-.05);
	\path [fill=white] (-.05,.85) rectangle (.95,.75);
	\path (1.2,0) rectangle (1.5,.8);
	\end{tikzpicture}
	\begin{tikzpicture}
	\draw [fill=lightgray] (0,0.8) arc[radius=0.3, start angle=180, end angle=360]--(0.8999999999999999,0.8)--(0.6,0)--(0.8999999999999999,0) arc[radius=0.44999999999999996, start angle=0, end angle=180]--(0.3,0)--(0.3,0.8)--(0,0.8);
	\path [fill=white] (-.05,.05) rectangle (.95,-.05);
	\path [fill=white] (-.05,.85) rectangle (.95,.75);
	\path (1.2,0) rectangle (1.5,.8);
	\end{tikzpicture}
	\begin{tikzpicture}
	\draw [fill=lightgray] (0,0.8)--(0.6,0)--(0.8999999999999999,0)--(0.6,0.8)--(0.8999999999999999,0.8)--(0.3,0)--(0,0)--(0.3,0.8)--(0,0.8);
	\path [fill=white] (-.05,.05) rectangle (.95,-.05);
	\path [fill=white] (-.05,.85) rectangle (.95,.75);
	\path (1.2,0) rectangle (1.5,.8);
	\end{tikzpicture}
	\begin{tikzpicture}
	\draw [fill=lightgray] (0,0.8)--(0.3,0)--(0,0) arc[radius=0.44999999999999996, start angle=180, end angle=0]--(0.6,0)--(0.6,0.8)--(0.8999999999999999,0.8) arc[radius=0.3, start angle=360, end angle=180]--(0,0.8);
	\path [fill=white] (-.05,.05) rectangle (.95,-.05);
	\path [fill=white] (-.05,.85) rectangle (.95,.75);
	\path (1.2,0) rectangle (1.5,.8);
	\end{tikzpicture}\\
	&\begin{tikzpicture}
	\draw [fill=lightgray] (0,0.8) arc[radius=0.3, start angle=180, end angle=360]--(0.8999999999999999,0.8)--(0.3,0)--(0,0) arc[radius=0.3, start angle=180, end angle=0]--(0.8999999999999999,0)--(0.3,0.8)--(0,0.8);
	\path [fill=white] (-.05,.05) rectangle (.95,-.05);
	\path [fill=white] (-.05,.85) rectangle (.95,.75);
	\path (1.2,0) rectangle (1.5,.8);
	\end{tikzpicture}
	\begin{tikzpicture}
	\draw [fill=lightgray] (0,0.8) arc[radius=0.44999999999999996, start angle=180, end angle=360]--(0.6,0.8)--(0.6,0)--(0.8999999999999999,0) arc[radius=0.3, start angle=0, end angle=180]--(0,0)--(0.3,0.8)--(0,0.8);
	\path [fill=white] (-.05,.05) rectangle (.95,-.05);
	\path [fill=white] (-.05,.85) rectangle (.95,.75);
	\path (1.2,0) rectangle (1.5,.8);
	\end{tikzpicture}
	\begin{tikzpicture}
	\draw [fill=lightgray] (0,0.8)--(0.3,0)--(0,0)--(0.6,0.8)--(0.8999999999999999,0.8)--(0.6,0)--(0.8999999999999999,0)--(0.3,0.8)--(0,0.8);
	\path [fill=white] (-.05,.05) rectangle (.95,-.05);
	\path [fill=white] (-.05,.85) rectangle (.95,.75);
	\path (1.2,0) rectangle (1.5,.8);
	\end{tikzpicture}
	\begin{tikzpicture}
	\draw [fill=lightgray] (0,0.8) arc[radius=0.44999999999999996, start angle=180, end angle=360]--(0.6,0.8)--(0.8999999999999999,0)--(0.6,0) arc[radius=0.3, start angle=0, end angle=180]--(0.3,0)--(0.3,0.8)--(0,0.8);
	\path [fill=white] (-.05,.05) rectangle (.95,-.05);
	\path [fill=white] (-.05,.85) rectangle (.95,.75);
	\path (1.2,0) rectangle (1.5,.8);
	\end{tikzpicture}
	\begin{tikzpicture}
	\draw [fill=lightgray] (0,0.8)--(0.8999999999999999,0)--(0.6,0)--(0.8999999999999999,0.8)--(0.6,0.8)--(0,0)--(0.3,0)arc[radius=.8, start angle=210, end angle=150]--(0,0.8);
	\path [fill=white] (-.05,.05) rectangle (.95,-.05);
	\path [fill=white] (-.05,.85) rectangle (.95,.75);
	\path (1.2,0) rectangle (1.5,.8);
	\end{tikzpicture}
	\begin{tikzpicture}
	\draw [fill=lightgray] (0,0.8)--(.6,.2)--(0.6,0)--(0.8999999999999999,0) arc[radius=0.2, start angle=0, end angle=90]--(.2,.2) arc[radius=.2, start angle=90, end angle=180]--(0.3,0)--(0.6,0.8)--(0.8999999999999999,0.8) arc[radius=0.3, start angle=360, end angle=180]--(0,0.8);
	\path [fill=white] (-.05,.05) rectangle (.95,-.05);
	\path [fill=white] (-.05,.85) rectangle (.95,.75);
	\path (1.2,0) rectangle (1.5,.8);
	\end{tikzpicture}
	\begin{tikzpicture}
	\draw [fill=lightgray] (0,0.8)--(0.8999999999999999,0)--(0.6,0)arc[radius=.8, start angle=-30, end angle=30]--(0.8999999999999999,0.8)--(0.3,0)--(0,0)--(0.3,0.8)--(0,0.8);
	\path [fill=white] (-.05,.05) rectangle (.95,-.05);
	\path [fill=white] (-.05,.85) rectangle (.95,.75);
	\path (1.2,0) rectangle (1.5,.8);
	\end{tikzpicture}
	\begin{tikzpicture}
	\draw [fill=lightgray] (0,0.8) arc[radius=0.2, start angle=180, end angle=270]--(.7,.6) arc[radius=.2, start angle=270, end angle=360]--(0.6,0.8)--(0.3,0)--(0,0) arc[radius=0.3, start angle=180, end angle=0]--(0.8999999999999999,0)--(.3,.6)--(0.3,0.8)--(0,0.8);
	\path [fill=white] (-.05,.05) rectangle (.95,-.05);
	\path [fill=white] (-.05,.85) rectangle (.95,.75);
	\path (1.2,0) rectangle (1.5,.8);
	\end{tikzpicture}\\
	&\begin{tikzpicture}
	\draw [fill=lightgray] (0,0.8)--(0.6,0)--(0.8999999999999999,0)--(0.6,0.8)--(0.8999999999999999,0.8)--(0,0)--(0.3,0) arc [radius=.8, start angle=210, end angle=150]--(0,0.8);
	\path [fill=white] (-.05,.05) rectangle (.95,-.05);
	\path [fill=white] (-.05,.85) rectangle (.95,.75);
	\path (1.2,0) rectangle (1.5,.8);
	\end{tikzpicture}
	\begin{tikzpicture}
	\draw [fill=lightgray] (0,0.8) arc[radius=0.3, start angle=180, end angle=360]--(0.8999999999999999,0.8)--(0.3,.2)--(.3,0)--(0,0) arc[radius=0.2, start angle=180, end angle=90]--(.7,.2) arc[radius=.2, start angle=90, end angle=0]--(0.6,0)--(0.3,0.8)--(0,0.8);
	\path [fill=white] (-.05,.05) rectangle (.95,-.05);
	\path [fill=white] (-.05,.85) rectangle (.95,.75);
	\path (1.2,0) rectangle (1.5,.8);
	\end{tikzpicture}
	\begin{tikzpicture}
	\draw [fill=lightgray] (0,0.8)--(0.3,0)--(0,0)--(0.8999999999999999,0.8)--(0.6,0.8) arc[radius=.8, start angle=30, end angle=-30]--(0.8999999999999999,0)--(0.3,0.8)--(0,0.8);
	\path [fill=white] (-.05,.05) rectangle (.95,-.05);
	\path [fill=white] (-.05,.85) rectangle (.95,.75);
	\path (1.2,0) rectangle (1.5,.8);
	\end{tikzpicture}
	\begin{tikzpicture}
	\draw [fill=lightgray] (0,0.8) arc[radius=0.2, start angle=180, end angle=270]--(.7,.6) arc[radius=.2, start angle=270, end angle=360]--(0.6,0.8)--(.6,.6)--(0,0)--(0.3,0) arc[radius=0.3, start angle=180, end angle=0]--(0.6,0)--(0.3,0.8)--(0,0.8);
	\path [fill=white] (-.05,.05) rectangle (.95,-.05);
	\path [fill=white] (-.05,.85) rectangle (.95,.75);
	\path (1.2,0) rectangle (1.5,.8);
	\end{tikzpicture}
	\end{align*}
	\item The diagrams containing 4 generators.
	\begin{align*}
	&
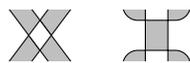
\begin{tikzpicture}
	\draw [fill=lightgray](0,0)--(.3,0)--(.9,.8)--(.6,.8)--(0,0);
	\draw [fill=lightgray] (0,.8)--(.3,.8)--(.9,0)--(.6,0)--(0,.8);
	\draw [fill=white] (.45,.6)--(.3,.4)--(.45,.2)--(.6,.4)--(.45,.6);
	\path [fill=white] (-.05,.05) rectangle (.95,-.05);
	\path [fill=white] (-.05,.85) rectangle (.95,.75);
	\path (1.2,0) rectangle (1.5,.8);
	\end{tikzpicture}
	\begin{tikzpicture}
	\draw [fill=lightgray] (0,0.8) arc[radius=.2, start angle=180, end angle=270]--(.7,.6) arc[radius=.2, start angle=270, end angle=360]--(0.6,0.8)--(0.6,0)--(0.8999999999999999,0) arc[radius=0.2, start angle=0, end angle=90]--(.2,.2) arc[radius=.2, start angle=90, end angle=180]--(0.3,0)--(0.3,0.8)--(0,0.8);
	\path [fill=white] (-.05,.05) rectangle (.95,-.05);
	\path [fill=white] (-.05,.85) rectangle (.95,.75);
	\path (1.2,0) rectangle (1.5,.8);
	\end{tikzpicture}
	\end{align*}
	\item The diagrams containing at least 5 generators.	
	\begin{align*}
	&\begin{tikzpicture}
	\path [fill=white] (0,0+.8+.4)--(.3,0+.8+.4)--(0,-.4+.8+.4)--(.3,-.4+.8+.4);
	\path [fill=lightgray](0,0)--(.3,0)--(.9,.8)--(.6,.8)--(0,0);
	\path [fill=lightgray] (0,.8)--(.3,.8)--(.9,0)--(.6,0)--(0,.8);
	\draw [fill=white] (.45,.6)--(.3,.4)--(.45,.2)--(.6,.4)--(.45,.6);
	\draw (0,0)--(.6,.8);
	\draw (.3,0)--(.9,.8);
	\draw (0,.8)--(.6,0);
	\draw (.3,.8)--(.9,0);
	\path (1.2,0) rectangle (1.5,.8);
	\path [fill=lightgray] (0,0)--(.3,0)--(0,-.4)--(.3,-.4);
	\draw (0,0)--(.3,-.4);
	\draw (.3,0)--(0,-.4);
	\end{tikzpicture}
	\begin{tikzpicture}
	\path [fill=white] (0,0+.8+.4)--(.3,0+.8+.4)--(0,-.4+.8+.4)--(.3,-.4+.8+.4);
	\path [fill=lightgray](0,0)--(.3,0)--(.9,.8)--(.6,.8)--(0,0);
	\path [fill=lightgray] (0,.8)--(.3,.8)--(.9,0)--(.6,0)--(0,.8);
	\draw [fill=white] (.45,.6)--(.3,.4)--(.45,.2)--(.6,.4)--(.45,.6);
	\draw (0,0)--(.6,.8);
	\draw (.3,0)--(.9,.8);
	\draw (0,.8)--(.6,0);
	\draw (.3,.8)--(.9,0);
	\path (1.2,0) rectangle (1.5,.8);
	\path [fill=lightgray] (.6+0,0)--(.6+.3,0)--(.6+0,-.4)--(.6+.3,-.4);
	\draw (.6+0,0)--(.6+.3,-.4);
	\draw (.6+.3,0)--(.6+0,-.4);
	\end{tikzpicture}
	\begin{tikzpicture}
	\path [fill=lightgray](0,0)--(.3,0)--(.9,.8)--(.6,.8)--(0,0);
	\path [fill=lightgray] (0,.8)--(.3,.8)--(.9,0)--(.6,0)--(0,.8);
	\draw [fill=white] (.45,.6)--(.3,.4)--(.45,.2)--(.6,.4)--(.45,.6);
	\draw (0,0)--(.6,.8);
	\draw (.3,0)--(.9,.8);
	\draw (0,.8)--(.6,0);
	\draw (.3,.8)--(.9,0);
	\path (1.2,0) rectangle (1.5,.8);
	\path [fill=lightgray] (.6+0,0+.8+.4)--(.6+.3,0+.8+.4)--(.6+0,-.4+.8+.4)--(.6+.3,-.4+.8+.4);
	\draw (.6+0,0+.8+.4)--(.6+.3,-.4+.8+.4);
	\draw (.6+.3,0+.8+.4)--(.6+0,-.4+.8+.4);
	\path [fill=white] (0,0)--(.3,0)--(0,-.4)--(.3,-.4);
	\end{tikzpicture}
	\begin{tikzpicture}
	\path [fill=lightgray](0,0)--(.3,0)--(.9,.8)--(.6,.8)--(0,0);
	\path [fill=lightgray] (0,.8)--(.3,.8)--(.9,0)--(.6,0)--(0,.8);
	\draw [fill=white] (.45,.6)--(.3,.4)--(.45,.2)--(.6,.4)--(.45,.6);
	\draw (0,0)--(.6,.8);
	\draw (.3,0)--(.9,.8);
	\draw (0,.8)--(.6,0);
	\draw (.3,.8)--(.9,0);
	\path (1.2,0) rectangle (1.5,.8);
	\path [fill=lightgray] (0,0+.8+.4)--(.3,0+.8+.4)--(0,-.4+.8+.4)--(.3,-.4+.8+.4);
	\draw (0,0+.8+.4)--(.3,-.4+.8+.4);
	\draw (.3,0+.8+.4)--(0,-.4+.8+.4);
	\path [fill=white] (0,0)--(.3,0)--(0,-.4)--(.3,-.4);
	\end{tikzpicture}
	\begin{tikzpicture}
	\path [fill=lightgray](0,0)--(.3,0)--(.9,.8)--(.6,.8)--(0,0);
	\path [fill=lightgray] (0,.8)--(.3,.8)--(.9,0)--(.6,0)--(0,.8);
	\draw [fill=white] (.45,.6)--(.3,.4)--(.45,.2)--(.6,.4)--(.45,.6);
	\draw (0,0)--(.6,.8);
	\draw (.3,0)--(.9,.8);
	\draw (0,.8)--(.6,0);
	\draw (.3,.8)--(.9,0);
	\path (1.2,0) rectangle (1.5,.8);
	\path [fill=lightgray] (0,0)--(.3,0)--(0,-.4)--(.3,-.4);
	\draw (0,0)--(.3,-.4);
	\draw (.3,0)--(0,-.4);
	\path [fill=lightgray] (.6+0,0+.8+.4)--(.6+.3,0+.8+.4)--(.6+0,-.4+.8+.4)--(.6+.3,-.4+.8+.4);
	\draw (.6+0,0+.8+.4)--(.6+.3,-.4+.8+.4);
	\draw (.6+.3,0+.8+.4)--(.6+0,-.4+.8+.4);
	\end{tikzpicture}
	\begin{tikzpicture}
	\path [fill=lightgray](0,0)--(.3,0)--(.9,.8)--(.6,.8)--(0,0);
	\path [fill=lightgray] (0,.8)--(.3,.8)--(.9,0)--(.6,0)--(0,.8);
	\draw [fill=white] (.45,.6)--(.3,.4)--(.45,.2)--(.6,.4)--(.45,.6);
	\draw (0,0)--(.6,.8);
	\draw (.3,0)--(.9,.8);
	\draw (0,.8)--(.6,0);
	\draw (.3,.8)--(.9,0);
	\path (1.2,0) rectangle (1.5,.8);
	\path [fill=lightgray] (.6+0,0)--(.6+.3,0)--(.6+0,-.4)--(.6+.3,-.4);
	\draw (.6+0,0)--(.6+.3,-.4);
	\draw (.6+.3,0)--(.6+0,-.4);
	\path [fill=lightgray] (0,0+.8+.4)--(.3,0+.8+.4)--(0,-.4+.8+.4)--(.3,-.4+.8+.4);
	\draw (0,0+.8+.4)--(.3,-.4+.8+.4);
	\draw (.3,0+.8+.4)--(0,-.4+.8+.4);
	\end{tikzpicture}
	\begin{tikzpicture}
	\path [fill=white] (0,0)--(.3,0)--(0,-.4)--(.3,-.4);
	\path [fill=white] (0,0+.8+.4)--(.3,0+.8+.4)--(0,-.4+.8+.4)--(.3,-.4+.8+.4);
	\draw [fill=lightgray] (0,0.8)--(0.8999999999999999,0)--(0.6,0)--(0.6,0.8)--(0.8999999999999999,0.8)--(0,0)--(0.3,0)--(0.3,0.8)--(0,0.8);
	\path [fill=white] (-.05,.05) rectangle (.95,-.05);
	\path [fill=white] (-.05,.85) rectangle (.95,.75);
	\path (1.2,0) rectangle (1.5,.8);
	\end{tikzpicture}
	\end{align*}
\end{itemize}

\bibliography{bibliography}
\bibliographystyle{amsalpha}

\end{document}